\documentclass[reqno,11pt]{amsart}
\usepackage{mathrsfs,graphicx,amsthm,amsfonts,amsxtra,xspace,array,cite,epic,float,ifpdf,bm,xcolor}
\usepackage[left=2.5cm]{geometry}
\usepackage{url}
\usepackage{enumerate}
\usepackage{amsmath,amssymb,amscd,mathrsfs}

\usepackage{ifthen}     

\newcommand\pubpri[2]{%
\ifthenelse{\equal{\version}{public}}%
{{#1}}%
{\marginpar{\scshape\small Pubpri Alert}{#2}}}
\newcommand\pubprinoalert[2]{%
\ifthenelse{\equal{\version}{public}}%
{{#1}}%
{#2}}

\newcommand\ignore[1]{}

\usepackage{color}
\providecommand\wantcolor{yes}   %
\definecolor{backgroundyellow}{cmyk}{.2,.1,.8,.2}
\definecolor{backgroundblue}{rgb}{0,0,1}
\definecolor{backgroundred}{rgb}{1,0,0}
\definecolor{backgroundmagenta}{cmyk}{0,1,0,0}
\ifthenelse{\equal{\wantcolor}{no}}{\renewcommand\color[1]{}}{}

\newcommand\mysubsubsection[1]{%
		\subsubsection{\sffamily\upshape\mdseries #1}}

\newcommand\mysss{\mysubsubsection}

\ifpdf
  \usepackage[
    pdftex,
    colorlinks,%
    linkcolor=blue,citecolor=red,urlcolor=blue,
    hyperindex,%
    plainpages=false,%
    bookmarksopen,%
    bookmarksnumbered%
  ]{hyperref} 
\else
  \usepackage{hyperref}
\fi
\newtheorem{annotation}{Annotation}[section]
\newtheorem{theorem}[annotation]{
		Theorem}
\newtheorem{lemma}[annotation]{
		Lemma}
\newtheorem{definition}[annotation]{
		Definition}
\newtheorem{corollary}[annotation]{
		Corollary}
\newtheorem{proposition}[annotation]{
		Proposition}
\newtheorem{example}[annotation]{
		Example}
		
\newcommand\bexample{\begin{example}\begin{rm}}
\newcommand\eexample{\end{rm}\hfill$\Box$\end{example}}
\newtheorem{examplenobox}[annotation]{
		Example}
\newcommand\bexamplenobox{\begin{examplenobox}\begin{rm}}
\newcommand\eexamplenobox{\end{rm}\end{examplenobox}}
\newtheorem{exercise}[annotation]{
		Exercise}
\newcommand\bexercise{\begin{exercise}\begin{rm}}
\newcommand\eexercise{\end{rm}\end{exercise}}
\newtheorem{notation}[annotation]{
		Notation}
\newcommand\bnotation{\begin{notation}\begin{rm}}
\newcommand\enotation{\end{rm}\end{notation}}

\newtheorem{remark}[annotation]{
		Remark}
\newcommand\bremark{\begin{remark}
\begin{upshape}}
\newcommand\eremark{\end{upshape}
\end{remark}}
\newcommand\bdefn{\begin{definition}
\begin{upshape}}
\newcommand\edefn{\end{upshape}
\end{definition}}
\newtheorem{caveat}[annotation]{
		Caveat}
\newcommand\bcaveat{\begin{caveat}
\begin{upshape}}
\newcommand\ecaveat{\end{upshape}
\end{caveat}}
\newenvironment{caveatstar}{
\par\noindent{\scshape\bfseries
  Caveat: }\begin{rm}}{\end{rm}\newline} 
\newcommand\bcaveatstar{\begin{caveatstar}}
\newcommand\ecaveatstar{\end{caveatstar}}
\newenvironment{myproof}{%
\par\noindent{\scshape
  Proof: }\begin{rm}}{\hfill$\Box$\end{rm}\newline} 
\newcommand\bmyproof{\begin{myproof}}
\newcommand\emyproof{\end{myproof}}
\newenvironment{myproofnobox}{%
\par\noindent{\scshape Proof:}\begin{rm}}{\end{rm}\hfill\newline}
\newcommand\bmyproofnobox{\begin{myproofnobox}}
\newcommand\emyproofnobox{\end{myproofnobox}}
\newenvironment{solution}{%
\par\noindent{\scshape Solution: }\begin{rm}}{\hfill$\Box$\end{rm}\newline}
\newenvironment{solutionnobox}{%
\par\noindent{\scshape Solution: }\begin{rm}}{\end{rm}}
\newcommand\bsolution{\begin{solution}\begin{rm}}
\newcommand\esolution{\end{rm}\end{solution}}
\newcommand\bsolutionnobox{\begin{solutionnobox}\begin{rm}}
\newcommand\esolutionnobox{\end{rm}\end{solutionnobox}}
\newcommand\bthm{\begin{theorem}}
\newcommand\ethm{\end{theorem}}
\newcommand\bcor{\begin{corollary}}
\newcommand\ecor{\end{corollary}}
\newcommand\blemma{\begin{lemma}}
\newcommand\elemma{\end{lemma}}
\newcommand\bprop{\begin{proposition}}
\newcommand\eprop{\end{proposition}}
\newcommand\beqn{\begin{equation}}
\newcommand\eeqn{\end{equation}}
\newcommand\beqnstar{\begin{equation*}}
\newcommand\eeqnstar{\end{equation*}}

\newcommand\mtitle[1]%
{\marginpar{\sffamily\raggedright\upshape\small #1}}

\providecommand\finalized{no}
\ifthenelse{\equal{\finalized}{yes}}{}{\marginparwidth=2cm}
\ifthenelse{\equal{\finalized}{yes}}%
{}%
{}
\ifthenelse{\equal{\finalized}{yes}}%
{\newcommand\checked[1]{}}%
{\newcommand\checked[1]{\marginpar{[{\ttfamily\upshape\tiny CHECKED: #1}]}}}
\ifthenelse{\equal{\finalized}{yes}}%
{\newcommand\spellchecked[1]{}}%
{\newcommand\spellchecked[1]{\marginpar{[{\ttfamily\upshape\tiny SPELLCHECKED: #1}]}}}

\providecommand\version{public}   
\ifthenelse{\equal{\version}{public}}%
{\newcommand\mcomment[1]{}}%
{\newcommand\mcomment[1]{\marginpar{{\raggedright\sffamily\upshape\small
\begin{spacing}{0.75} #1\end{spacing}}}}}
\ifthenelse{\equal{\version}{public}}%
{\newcommand\fcomment[1]{}}%
{\newcommand\fcomment[1]{\footnote{#1}}}
\ifthenelse{\equal{\version}{public}}%
{\newcommand\comment[1]{}}%
{\newcommand\comment[1]{{\small #1}}}

\newcommand{\be}{\begin{enumerate}}
\newcommand{\ee}{\end{enumerate}}
\newcommand{\beq}{\begin{equation}}
\newcommand{\eeq}{\end{equation}}
\newcommand{\beqs}{\begin{equation*}}                     
\newcommand{\eeqs}{\end{equation*}}
\newcommand{\complex}{\mathbb{C}}

\newcommand{\nat}{\mathbb{N}}

\newcommand{\integers}{\mathbb{Z}}

\newcommand{\bin}[2]{\left[{#1}\atop{#2}\right]}

\newcommand\pgr{{\rm gr}}
\newcommand{\gr}{\mathpzc{gr}}

\newcommand{\cal}{\mathcal}

\newcommand{\bN}{\mathbb{N}}
\newcommand{\bn}{\mathbb{N}}
\newcommand{\bxu}{{\underline{\bx}}}
\newcommand{\bxur}{\bxu_r}
\newcommand{\bxurm}{\bxu_{r-1}}
\newcommand{\ul}{{\underline{\ell}}}
\newcommand{\ud}{{\underline{d}}}

\newcommand{\um}{{\underline{m}}}
\newcommand{\un}{{\underline{n}}}
\newcommand{\uk}{{\underline{k}}}
\newcommand{\ue}{{\underline{e}}}

\newcommand{\up}{{\underline{\mathbf{p}}}}
\newcommand{\uq}{{\underline{\mathbf{q}}}}
\newcommand{\uu}{{\underline{\mathbf{u}}}}

\newcommand{\us}{{\underline{\mathbf{s}}}}
\newcommand{\bF}{{\mathbf{F}}}
\newcommand{\bos}{{\mathbf{s}}}
\newcommand{\bop}{\mathbf{p}}
\newcommand{\boq}{\mathbf{q}}
\newcommand{\boi}{{\mathbf{i}}}

\newcommand{\boe}{{\mathbf{e}}}

\newcommand{\bu}{{\mathbf{U}}}
\newcommand{\bi}{{\mathbf{I}}}
\newcommand{\bx}{{\mathbf{x}}}
\newcommand{\bpr}{\mathbf{p}\mathbf{r}^-}

\def \eql {\,{\vartriangleright}\,}
\def \eqd {\,{\blacktriangleright}\,}

\newcommand{\kbar}{{\overline{k}}}

\newcommand{\pL}{\mathscr{L}}
\newcommand{\pv}{\mathpzc{v}}

\newcommand{\pW}{\mathpzc{W}}

\DeclareMathAlphabet{\mathpzc}{OT1}{pzc}{m}{it}

\renewcommand\omega{\varpi}

\newcommand{\csa}{\mathfrak{h}}

\newcommand{\lien}{\mathfrak{n}}
\newcommand\lieg{\mathfrak{g}}

\newcommand{\etaseq}{\underline{\eta}}

\newcommand\loneseq{\lseq^1}

\newcommand\eseq{\underline{\eta}}
\newcommand{\etwoseq}{\eseq^2}

\newcommand{\eoneseq}{\underline{\eta}^1}
\newcommand{\erseq}{\underline{\eta}^r}

\newcommand\lrseq{\lseq^r}
\newcommand{\lrmseq}{\lseq^{r-1}}

\newcommand\pattern{\mathcal{P}}

\newcommand{\lieh}{\csa}
\newcommand\pop{\mathfrak{P}}

\newcommand\lseq{\underline{\lambda}}

\newcommand\popset{\mathbb{P}}

\newcommand\currlieg{\lieg[t]}
\newcommand\curralg\currlieg

\newcommand{\lieb}{\mathfrak{b}}

\title[Bases for local Weyl modules in type $C$]
{Bases for local Weyl modules in type $C$}
\author{B.~Ravinder}
\address{Chennai Mathematical Institute\\ Plot H1, SIPCOT IT Park, Siruseri\\ Kelambakkam 603103,
India}
\email{bravinder@cmi.ac.in}

\begingroup\makeatletter\ifx\SetFigFont\undefined%
\gdef\SetFigFont#1#2#3#4#5{
  \reset@font\fontsize{#1}{#2pt}
  \fontfamily{#3}\fontseries{#4}\fontshape{#5}
  \selectfont}
\fi\endgroup
\begin{document}
\allowdisplaybreaks
\numberwithin{equation}{section}
\begin{abstract}
We give a Poincare-Birkhoff-Witt type basis for local Weyl modules of the current algebra of type $C$. As a consequence, we get a fermionic character formula for these modules.
\end{abstract}

\maketitle
\section{Introduction}
Let $\lieg$ be a finite-dimensional complex simple Lie algebra and $\lieg[t]=\lieg\otimes\mathbb{C}[t]$ be the corresponding current algebra. 
The degree grading on $\complex[t]$ gives a natural $\mathbb{Z}$-grading on $\lieg[t]$ and makes it a graded Lie algebra.
Local Weyl modules, introduced by Chari and Pressley \cite{CL} are important finite-dimensional graded
$\lieg[t]$--modules.
To each dominant integral weight $\lambda$ of~$\lieg$,  there corresponds a local Weyl module $W(\lambda)$. 
The module $W(\lambda)$ is universal among the finite-dimensional $\lieg[t]$--modules generated by a highest weight vector of weight $\lambda$; any such module is uniquely a quotient of $W(\lambda)$. The zeroth graded piece of $W(\lambda)$ is the irreducible representation $V(\lambda)$ of $\lieg$ and they both have the same $\lieg$-weights.

The work of 
Chari, Pressley, and Loktev \cite{CL, CP}  gives  nice monomial bases for  the local Weyl modules when $\lieg$ is of type $A$.
Recently, in \cite{RRV2}, an elegant combinatorial description for its parametrizing set is given, namely, as the
set of {\em partition overlaid patterns} (POPs).  Classical Gelfand-Tsetlin (GT) patterns \cite{GT} index a basis of $V(\lambda)$. It is shown in \cite{RRV2} that extending them by a partition overlay produces a parametrizing set for a basis of $W(\lambda)$.

Bases for local Weyl modules are not known beyond type $A$.  However  there are generalizations of GT patterns for all classical Lie algebras (see \cite{BZ, M}).
It is now natural to ask whether we can define overlays on these generalized patterns and use them to parametrize a monomial basis of local Weyl modules in these types.  In this paper we answer this question in the case  of type $C$, i.e., when $\lieg=\mathfrak{sp}_{2r}$, the symplectic Lie algebra.
We consider type $C$ for two reasons: (a) the generalized GT patterns are simpler (b) the local Weyl module associated to a fundamental weight is irreducible as a module for the underlying simple Lie algebra.

More precisely, for the Lie algebra $\lieg$ of type $C$, we introduce the notion of a {\em partition overlaid pattern} (POP), and corresponding to each POP, we define a monomial. We then prove that the  monomials corresponding to  POPs with bounding sequence $\lambda$, form a basis for the local Weyl module $W(\lambda)$ (Theorem~\ref{weylbasis}).  Analogous to \cite{CL},  the proof involves constructing a filtration for the local Weyl modules, studying the associated graded spaces, and uses an induction on the rank of the Lie algebra $\lieg$.  As an auxiliary step, we use the intermediate subalgebras $\mathfrak{sp}_{2r-1} (\mathfrak{sp}_{2r-2}\subset\mathfrak{sp}_{2r-1}\subset\mathfrak{sp}_{2r})$. These intermediate subalgebras and their representations have been studied by Gelfand-Zelevinsky \cite{GZ}, Proctor \cite{P}, and Shtepin \cite{S}.

As an application of our main theorem, we obtain  a fermionic formula for the character of $W(\lambda)$ (Corollary~\ref{fermionic_char}). Further since the basis for the  zeroth graded piece of $W(\lambda)$ corresponds to the empty overlay, we also get a monomial basis for $V(\lambda)$ indexed by generalized GT patterns. The exponents occurring in the monomials are certain differences of adjacent entries in the pattern. The  type $C$ classical GT basis \cite{M} has an analogous description in which these same exponents occur, with the Mickelsson lowering operators $z_{k,i}, z_{i,-k}$ (see \cite[Theorem~3.5]{M}) in place of the Chevalley generators $x^-_{i,k}, x^-_{i,\overline{k}}$ (see Corollary~\ref{irrbasis}).
We also obtain another proof of the formula for the dimension of the local Weyl module (Corollary~\ref{dimformula}). Other known proofs use path models, crystal bases, and the theory of global basis (see \cite{N}). 

In \cite{N}, Naoi has shown that the local Weyl module for $\mathfrak{sp}_{2r}[t]$ admits a filtration by $\mathfrak{sp}_{2r}[t]$-modules whose successive quotients are Demazure modules.
As a corollary to our main results, we prove that the local Weyl module for $\mathfrak{sp}_{2r}[t]$ admits a filtration by $\mathfrak{sp}_{2r-2}[t]$-modules whose successive quotients are local Weyl modules (Corollary~\ref{weylfiltration}). This result is analogous to the well-known result that an irreducible module for $\mathfrak{sp}_{2r}$ admits a filtration by $\mathfrak{sp}_{2r-2}$-modules whose successive quotients are irreducible modules \cite{Z}. 

\subsection*{Acknowledgements}
The author thanks K. N. Raghavan and S. Viswanath for many helpful discussions.
He acknowledges support from DST under the INSPIRE Faculty scheme.
He also acknowledges received fellowship from  Infosys Foundation.
\section{Notation and Preliminaries}\label{notn}
Throughout the paper, $\mathbb{Z}$ denotes the set of integers, $\mathbb{N}$ the set of non-negative integers, $\mathbb{N}_+$ the set of positive integers,  
   $\complex$  the field of complex numbers, 
   $\complex[t]$ the polynomial ring,   and $\mathbf{U}(\mathfrak{a})$ the universal
enveloping algebra corresponding to a complex Lie algebra $\mathfrak{a}$.
\subsection{}
Let $\mathfrak{gl}_{2r}$ be the Lie algebra of $2r\times 2r$ complex matrices whose rows and columns are enumerated by the indices $1, 2,\ldots, r,$ $-r,\ldots,-1$.
Set $I=\{1,2,\ldots,r\}$.
The symplectic Lie algebra $\mathfrak{g}=\mathfrak{sp}_{2r}$ is the space  of matrices $(a_{i,j})\in\mathfrak{gl}_{2r}$ for which 
$$a_{i,j}=-\textup{sgn}\,i\cdot\textup{sgn}\,j\cdot a_{-j,-i},\qquad\forall\,\,i,j\in
\pm  I.$$ For $i,j\in\pm I$, let $E_{i,j}$  be the $2r\times 2r$ matrix with $1$  in the $(i,j)^{th}$
position
and $0$ 
elsewhere. 
 Fix the Cartan subalgebra $\csa$ of $\lieg$ spanned by the elements $E_{i,i}-E_{-i,-i}, i\in I$. 
 Define the elements $\varepsilon_1,\ldots,\varepsilon_r\in\csa^*$ by the equalities $\langle\varepsilon_i, \,E_{j,j}-E_{-j,-j}\rangle=\delta_{i,j},\,\,\forall\,\,i,j\in I$. Then the root system $R$ of $\lieg$ with respect to $\csa$ is given by
$$ R=\{\pm\varepsilon_i\pm\varepsilon_j:1\leq i< j\leq r\}\cup\{\pm 2\varepsilon_i:1\leq i\leq r\}.$$
 
 Fix $\alpha_1=\varepsilon_1-\varepsilon_2, \ldots, \alpha_{r-1}=\varepsilon_{r-1}-\varepsilon_r, \alpha_r=2\varepsilon_r$, a set of simple roots in $R$, then the set $R^+$  of positive roots is given by
$$\alpha_{i,j}=\alpha_i+\cdots+\alpha_j, \qquad \alpha_{i,\overline{j}}=\alpha_i+\cdots+\alpha_r+\alpha_{r-1}+\cdots+\alpha_j,\qquad 1\leq i \leq j \leq r.$$ 
Note that $\alpha_{i,r}=\alpha_{i,\overline{r}},\,\,\forall\,\,i\in I$. Set $R^-=-R^+$. 
Let $Q=\sum_{i\in I}\mathbb{Z}\alpha_i$ be the root lattice of $\lieg$ and set $Q^+=\sum_{i\in I}\mathbb{N}\alpha_i$.
 Define the subalgebras $\mathfrak{n}^{\pm}$ and $\mathfrak{b}^\pm$ of $\lieg$ by 
$$\mathfrak{n}^{\pm}=\bigoplus_{1\leq i \leq j < r} \complex x_{i,j}^{\pm}\bigoplus_{1\leq i \leq j \leq r} \complex x_{i,\overline{j}}^{\pm},\qquad \mathfrak{b}^\pm=\lien^\pm\oplus\lieh,$$
where \begin{gather*}
x_{{i,j-1}}^+=E_{i,j}-E_{-j,-i}, \quad x_{{i,j-1}}^-=E_{j,i}-E_{-i,-j}, \qquad
x_{{i,\overline{j}}}^+=E_{i,-j}+E_{j,-i}, \quad x_{{i,\overline{j}}}^-=E_{-j,i}+E_{-i,j},
\end{gather*}
for all $1\leq i<j\leq r$, and $x_{i,\overline{i}}^+=E_{i,-i},\,  x_{i,\overline{i}}^-=E_{-i,i},$ for all $i\in I.$
Set $$x_{i,r}^\pm=x_{i,\overline{r}}^\pm, \quad\forall\,\, i\in I, \qquad\alpha_{i,j}^\vee=[x_{i,j}^+,\,x_{i,j}^-], \qquad \textup{and} \qquad \alpha_{i,\overline{j}}^\vee=[x_{i,\overline{j}}^+, \,x_{i,\overline{j}}^-], \quad\forall \,\, 1\leq i\leq j\leq r.$$
 Now, we have the triangular decomposition: $\lieg= \mathfrak{n}^{-}\oplus \csa \oplus \mathfrak{n}^{+}.$

Let $\omega_i=\varepsilon_1+\cdots+\varepsilon_i, i\in I,$ be the set of fundamental weights 
of $\lieg$.
Let $P=\sum_{i\in I}\mathbb{Z}\varpi_i$ be the weight lattice of $\lieg$ and  $P^+=\sum_{i\in I}\nat\varpi_i$ be the set of dominant integral weights of $\lieg$.
For $\lambda=\sum_{i=1}^r m_i\omega_i\in P^+,$ the corresponding finite-dimensional
irreducible $\lieg$--module $V(\lambda)$ is the cyclic $\lieg$--module generated by 
$v_\lambda$ with defining relations:
$$\mathfrak{n}^+\,v_\lambda=0, \qquad h\,v_\lambda=\langle\lambda, \,h\rangle v_\lambda, \qquad (x_{i,i}^-)^{m_i+1}\, v_\lambda=0,$$
for all $i\in I$ and $h\in \csa$. 
Given an  $\lieh$--module $V$, we say that an element $\mu\in\csa^*$ is a weight of $V$ if there exists a non-zero element $v\in V$ such that $h\cdot v=\langle\mu,\,h\rangle v$, for all $h\in\csa$. Let $V_{\mu}$ denote the weight space of  $V$ of weight $\mu$.

For $\lambda=m_1\varpi_1+\cdots+m_r\varpi_r\in P^+$, we associate an integer tuple
$\underline{\lambda}=(\lambda_1\geq\lambda_2\geq\cdots\geq \lambda_r\geq \lambda_{r+1}=0)$ by setting $\lambda_i=m_i+\cdots+m_r$. Conversely, given such integer tuple, we associate an element $\lambda$ of $P^+$ by $\lambda=\lambda_1\varepsilon_1+\cdots+\lambda_r\varepsilon_r=m_1\varpi_1+\cdots+m_r\varpi_r$, where $m_i=\lambda_i-\lambda_{i+1}$. Thus the elements of $P^+$ are identified with the integer tuples of the form $(\lambda_1\geq\lambda_2\geq\cdots\geq \lambda_r\geq\lambda_{r+1}=0)$.
\subsection{}
The intermediate Lie subalgebra $\mathfrak{g}_{r-1/2}=\mathfrak{sp}_{2r-1}$ is singled out from $\lieg$ by the conditions $a_{i,r}=0,$ $a_{-r,i}=0,\,\,\,\forall \,\,i\in\pm I$, and the subalgebra $\mathfrak{g}_{r-1}=\mathfrak{sp}_{2r-2}$ from $\lieg_{r-1/2}$ by the conditions $a_{r,j}=0,$ $a_{j,-r}=0,\,\,\,\forall \,\,j\in\pm I.$ The root system $R_{r-1/2}$ of the Lie algebra $\lieg_{r-1/2}$ is given by  
$$R_{r-1/2}=\{\alpha\in R_r: (\alpha, \varepsilon_r)\geq 0\}.$$
Set $R_{r-1/2}^{\pm}=R^{\pm}\cap R_{r-1/2}$,  $\mathfrak{n}_{r-1/2}^{\pm}=\lien^\pm\cap\lieg_{r-1/2}$, and $\mathfrak{b}_{r-1/2}^\pm=\lieb^\pm\cap\lieg_{r-1/2}.$ 
Observe that 
$$\mathfrak{n}^+_{r-1/2}=\bigoplus_{1\leq i \leq j < r-1} \complex x_{i,j}^{+}\bigoplus_{1\leq i \leq j \leq r} \complex x_{i,\overline{j}}^{+}, \ \ \ \mathfrak{n}^-_{r-1/2}=\bigoplus_{1\leq i \leq j < r} \left(\complex x_{i,j}^{-}\oplus\complex x_{i,\overline{j}}^{-}\right), \ \ \  \mathfrak{b}_{r-1/2}^\pm=\lien_{r-1/2}^\pm\oplus\csa_{r-1}, $$
where $\csa_{r-1}$ is the Lie subalgebra of $\lieh$ spanned by the elements $\alpha_{i,i}^\vee, \, 1\leq i<r.$
Now, we have the triangular decomposition: $\lieg_{r-1/2}= \mathfrak{n}_{r-1/2}^{-}\oplus \csa_{r-1} \oplus \mathfrak{n}_{r-1/2}^{+}.$

For $\eta=\sum_{i=1}^rn_i\omega_i\in P^+$, we recall from \cite[Theorem 2]{S}, the corresponding finite-dimensional  highest weight  $\lieg_{r-1/2}$--module $\pL(\eta)$. It is the cyclic $\lieg_{r-1/2}$--module generated by $\pv_\eta$ with defining relations:
$$\mathfrak{n}^+_{r-1/2}\,\pv_\eta=0, \qquad h\,\pv_\eta=\langle\eta, \,h\rangle\pv_\eta, \qquad (x_{i,i}^-)^{n_i+1}\,\pv_\eta=0, \qquad (x^-_{r-1, \overline{r-1}})^{n_{r-1}+n_r+1}\,\pv_\eta=0,$$
for all $1\leq i<r$ and $h\in \csa_{r-1}$. 
\begin{theorem}\cite[Proposition 4.3]{S}\label{theoremofS}
\begin{enumerate}
\item \label{theoremofSp1}For $\lambda\in P^+$, the irreducible $\lieg$--module $V(\lambda)$ has a filtration by
$\lieg_{r-1/2}$--modules whose successive  quotients are $\pL(\eta_1\varepsilon_1+\cdots+\eta_r\varepsilon_r)$, where $\lambda_i\geq \eta_i\geq\lambda_{i+1}, \,\,\forall\,\,1\leq i\leq r$.
\item \label{theoremofSp2}For $\eta\in P^+$, the $\lieg_{r-1/2}$-- module $\pL(\eta)$ has a filtration by $\lieg_{r-1}$--modules whose successive quotients are $V(\nu_1\varepsilon_1+\cdots+\nu_{r-1}\varepsilon_{r-1}),$ where $\eta_i\geq\nu_i\geq\eta_{i+1},\,\,\forall\,\,1\leq i<r$.
In particular, for $\eta, \mu\in P^+$, the $\lieg_{r-1/2}$--modules  $\pL(\eta)$ and $\pL(\mu)$ are not isomorphic, unless $\eta=\mu$.
\end{enumerate}
\end{theorem}

\subsection{}
The current algebra $\mathfrak{a}[t]$  associated to a Lie algebra $\mathfrak{a}$ is defined by $\mathfrak{a}[t]=\mathfrak{a} \otimes \complex[t]$, with the Lie bracket
\[ [x \otimes t^m,  \,y \otimes t^n]=[x,\, y]\otimes t^{m+n} \quad \forall \,\, x, y \in  \mathfrak{a},\,\, m, n \in \,\mathbb{N}.\]
We regard $\mathfrak{a}$ as a subalgebra of $\mathfrak{a}[t]$ by mapping $x\mapsto x\otimes 1$ for $x\in\mathfrak{a}$. We denote the maximal ideal of $\mathfrak{a}[t]$ spanned by elements $x\otimes t^m, x\in\mathfrak{a}, m\in\mathbb{N}_+$ by $\mathfrak{a}t[t].$
The degree grading on $\complex[t]$ gives a natural $\mathbb{N}$-grading on $\mathbf{U}(\mathfrak{a}[t])$: the element $(x_1\otimes t^{m_1})\cdots (x_k\otimes t^{m_k})$, for $x_i\in\mathfrak{a}, m_i\in \mathbb{N},$ has grade $m_1+\cdots+m_k.$
A graded $\mathfrak{a}[t]$--module is a $\mathbb{Z}$-graded vector space
$M=\bigoplus_{n\in\mathbb{Z}} M[n]$ such that
\[ (\mathfrak{a}\otimes t^m) M[n]\subset M[m+n], \quad \forall\,\, m\in\mathbb{N},\, n\in\mathbb{Z}.\]

Given $\lambda=\sum_{i=1}^r m_i\omega_i\in P^{+}$, the {\em local Weyl module\/} $W(\lambda)$ is the cyclic  $\lieg[t]$--module with generator $w_\lambda$ and relations (see~\cite[\S1.2.1]{CL}):
\begin{equation}\label{weyldefreln} \mathfrak{n}^{+}[t]\,w_{\lambda}=0, \qquad (h\otimes t^{s})\,w_{\lambda}=\langle\lambda, \,h\rangle\delta_{s,0}\,w_\lambda,\qquad (x_{i,i}^-)^{m_i+1}\,w_{\lambda}=0,\end{equation}
for all $i\in I, h\in\csa,$ and $ s\in\nat$.
We declare the grade of $w_\lambda$ to be zero. Since the defining relations of $W(\lambda)$ are graded, it follows
that $W(\lambda)$ is a graded $\lieg[t]$--module. 

The proof of the following theorem is analogous to that in  \cite[\S \S 1-2]{CP}.
\begin{theorem}\label{theorem_uni1} For all
$\lambda\in P^+$, the local Weyl module $W(\lambda)$ is
finite--dimensional. Moreover, any finite--dimensional $\lieg[t]$--module $V$ generated by an element $v\in V$ and satisfying the
relations
\begin{equation}\label{weylthmreln1}
\mathfrak{n}^{+}[t]\,v=0, \qquad (h\otimes t^{s})\,v=\langle\lambda, \,h\rangle\delta_{s,0}\,v,\,\,\forall\,\,h\in\csa, s\in\nat,
\end{equation}
is a quotient of $W(\lambda)$.
\end{theorem}

\subsection{}
For  $\lambda\in P^+$ and $s\in\mathbb{N}$, the subspace $W(\lambda)[{s}]$ of grade $s$  of the local Weyl module $W(\lambda)$ 
is a $\lieg$-submodule, and we have the following weight space 
decomposition for $W(\lambda)$:
$$
W(\lambda)= \bigoplus_{(\mu,s)\in P\times \mathbb{N}} W(\lambda)_{\mu,s},
$$
where $W(\lambda)_{\mu,s}:=\{w\in W(\lambda)[s]: h\,w=\langle\mu,\,h\rangle w,\,\forall\,h\in\csa\}.$ 
The graded character $\textup{ch}_{q} W(\lambda)$ of $W(\lambda)$  is defined by
\beq
\textup{ch}_{q} W(\lambda):=\sum_{(\mu,s)\in P\times \mathbb{N}} \dim W(\lambda)_
{\mu,s}\, e^{\mu}\,q^{s} \quad\in\integers[P][q].
\eeq
Here  $\integers[P]$ denotes the integral group ring of $P$ with basis $e^\mu, \mu\in P$.

\section{The main result}
\subsection{} The goal of this subsection is to introduce {\em partition overlaid patterns (POPs)} in type $C$. 
\subsubsection{Partitions} A {\em partition} is a non-decreasing sequence of non-negative integers. For a partition $\bos=(\bos(1)\le\cdots\le\bos(\ell))$, set    $|\bos|=\bos(1)+\cdots+\bos(\ell)$.
\subsubsection{Partition fitting into a rectangle} Let $\ell, \ell^\prime$ be non-negative integers. We say that a partition $\bos$  {\em fits into a rectangle} $(\ell, \ell^\prime)$, if $\bos=(\bos(1)\le\cdots\le\bos(\ell))\in\bn^\ell$ and $\bos(\ell)\le \ell'$. Note that the number of partitions that fit into a rectangle $(\ell,\ell')$ is $\binom{\ell+\ell'}{\ell}$.
\subsubsection{Generalized Gelfand-Tsetlin  patterns }\label{s:patterns}
A {\em generalized  Gelfand-Tsetlin (GT) pattern} of type $C$ (or just {\em pattern}) $\mathcal{P}^r$  is an array of
integral row vectors $\eoneseq,\loneseq, \etwoseq, \ldots, \lrmseq, \erseq, \lrseq$:
\begin{align}
&\ \ \quad\qquad\qquad\qquad\qquad\qquad\eta^1_1 \nonumber\\
&\qquad\qquad\qquad\qquad\qquad\lambda^1_1 \nonumber\\
&\qquad\qquad\qquad\qquad\eta^2_1\qquad\eta^2_2\nonumber\\
&\quad\quad\qquad\qquad\cdots\qquad\cdots\qquad\cdots\nonumber\\
&\qquad\qquad\lambda^{r-1}_1\ \quad\qquad \cdots \ \qquad\lambda^{r-1}_{r-1}\nonumber\\
&\qquad\eta^r_1\qquad\eta^{r}_2\qquad\cdots \qquad\eta^{r}_{r-1}\qquad\eta^r_r\nonumber\\
&\lambda^{r}_1\qquad\lambda^{r}_2
\qquad \cdots\qquad \lambda^r_{r-1}\qquad\lambda^{r}_{r}\nonumber
\end{align}
subject to the following conditions:
\[\lambda^j_i\ge \eta^j_i\ge \lambda^j_{i+1}, \quad\forall\,\,1\le i\le j\le r, \qquad \textup{and}\qquad\eta^{j+1}_i\geq\lambda^{j}_i\geq\eta^{j+1}_{i+1},\quad\forall\,\,1\leq i\leq j< r, \] where $\lambda^j_{j+1}=0$. 
The last sequence ${\lseq}^{r}$ of the pattern $\mathcal{P}^r$ is its
 {\em bounding sequence}.    
 \subsubsection{Restricted patterns }\label{s:respatterns}
A {\em restricted pattern} $\mathcal{P}^{r-1/2}$  is an array of
integral row vectors \\ $\eoneseq,\loneseq, \etwoseq, \ldots, \lrmseq, \erseq$ 
in which the following conditions hold:
\[\lambda^j_i\ge \eta^j_i\ge \lambda^j_{i+1} \qquad \textup{and}\qquad\eta^{j+1}_i\geq\lambda^{j}_i\geq\eta^{j+1}_{i+1},\quad\forall\,\,1\leq i\leq j< r.\ \]
The last sequence ${\eseq}^{r}$ of the restricted pattern $\mathcal{P}^{r-1/2}$ is its
 {\em bounding sequence}.

\subsubsection{The weight of a pattern} The {\em weight} of $\pattern^r:\eoneseq,\loneseq, \etwoseq, \ldots, \lrmseq, \erseq, \lrseq$ is defined by
$$a_1\varepsilon_1+a_2\varepsilon_2+\cdots+a_{r}\varepsilon_{r} \in\csa^*, \quad\textup{where}\quad a_j=2\sum_{i=1}^j\eta^j_i-\sum_{i=1}^{j}\lambda^j_i-\sum_{i=1}^{j-1}\lambda^{j-1}_i.$$

 \subsubsection{Differences corresponding to a pattern} Let $\pattern^{r-1/2}:\eoneseq,\loneseq, \etwoseq, \ldots, \lrmseq, \erseq$ be a restricted pattern  and $\pattern^{r}:\pattern^{r-1/2}, \lrseq$ be a pattern . The {\em differences} 
 $\ell_{i,\overline{j}},$ $\ell'_{i,\overline{j}},$ for $1\leq i\leq j\leq r$;  and $\ell_{i,j}$, $\ell^\prime_{i,j}$, for $1\leq i\leq j< r$,  corresponding to $\pattern^r$ are given by
$$\ell_{i,\overline{j}}:=\lambda^j_i-\eta^j_i, \quad   \ell'_{i,\overline{j}}:=\eta^j_i-\lambda^j_{i+1}; \qquad\ell_{i,j}:=\eta_{i}^{j+1}-\lambda^j_{i}, \quad \ell^\prime_{i,j}:=\lambda_{i}^j-\eta^{j+1}_{i+1}.$$ 
We shall also call the differences $\ell_{i,\overline{j}}$, $\ell'_{i,\overline{j}}$, $\ell_{i,j}$, $\ell^\prime_{i,j}$ corresponding to $\pattern^r$, for $1\leq i\leq j< r$,  as the differences corresponding to the restricted pattern $\pattern^{r-1/2}$.

\subsubsection{Partition overlaid patterns (POPs) } \label{s:pops}
A {\em partition overlaid pattern (POP) } in type $C$ consists of the following data
\begin{enumerate}
\item
a generalized GT pattern $\pattern^r$ of type $C$,
\item
for every pair $(i,j)$ of integers with $1\leq i\leq j\leq r$, 
a partition $\bos_{i,\overline{j}}$ that fits into the rectangle $(\ell_{i,\overline{j}}, \ell^\prime_{i,\overline{j}})$, 
\item
for every pair $(i,j)$ of integers with $1\leq i\leq j<r$, 
a partition $\bos_{i,j}$ that fits into the rectangle $(\ell_{i,j}, \ell^\prime_{i,j})$.
\end{enumerate}
For $\lambda\in P^+$, let $\popset^r_{\lambda}$ denote the set  of POPs  with bounding sequence $\lseq$. 
The {\em weight} of a POP is just the weight of the underlying pattern. The number of boxes in a POP $\pop^r$ is the sum $$\sum_{1\leq i\leq j\leq r}|\bos_{i,\overline{j}}
|+\sum_{1\leq i\leq j<r}|\bos_{i,j}|$$
of the number of boxes in each of its constituent partitions. It is denoted by $|\pop^r|$.
\subsubsection{Restricted partition overlaid patterns  } \label{s:respops}
A {\em restricted partition overlaid pattern } consists of a restricted pattern $\pattern^{r-1/2}$, and for every pair $(i,j)$ of integers with $1\leq i\leq j< r$, 
two partitions $\bos_{i,\overline{j}}$ and $\bos_{i,j}$ such that $\bos_{i,\overline{j}}$ fits into the rectangle $(\ell_{i,\overline{j}}, \ell^\prime_{i,\overline{j}})$ and  $\bos_{i,j}$ that fits into the rectangle $(\ell_{i,j}, \ell^\prime_{i,j})$.
For $\eta\in P^+$, let $\popset^{r-1/2}_{\eta}$ denote the set  of restricted POPs  with bounding sequence $\eseq$. 

\subsection{} In this subsection, we state the main result and give some applications. 
\subsubsection{}
 Let $\bF$ denote
 the set of all pairs
$(\ell,\bos)$ which satisfy: $$\ell\in\bn,\ \
 \bos=(\bos(1)\le \cdots \le \bos(\ell))\in\bn^\ell,$$
including the pair $(0, \emptyset)$, where $\emptyset$ is the
empty partition. Given an element $((\ell_1,\bos_1),\ldots,(\ell_j,\bos_j))$ of $\bF^j$, $j>0$, we let $(\ul,\us)$, be the
pair of $j$--tuples,   $\ul=(\ell_1,\cdots,\ell_j)$
and $\us=(\bos_1,\cdots,\bos_j)$. 
For $m\in\bn$,  let
 \begin{equation}\label{defnbf} \bF(m)=\left\{(\ell,\bos)\in\bF: m\geq \ell,\,\, \bos \,\,\textup{fits into the rectangle}\,\, ( \ell, m-\ell)\right\}\end{equation}
For $\lambda = \sum_{i=1}^r m_i \omega_i\in P^+ $  and $1\le j\le
r$, let
\begin{equation}\label{defnbfj} \bF^j(\lambda)=\left\{(\ul,
\us)\in\bF^j: (\ell_i, \bos_i) \in \bF(m_i),\ \ 1 \le i \le
j\right\}.\end{equation}
Observe that $$(\ul,
\us)\in\bF^j(\lambda)\Longrightarrow \lseq-\ul=(\lambda_1-\ell_1\geq\cdots\geq\lambda_j-\ell_j\geq\lambda_{j+1}\geq\cdots\geq\lambda_r\geq\lambda_{r+1}=0)\in P^+.$$

\subsubsection{} Given $(\ell,\bos)\in\bF$, let $\bx^\pm_{i,\overline{j}}(\ell,\bos)$ for $1\le i\le j\le r$, and $\bx^\pm_{i,j}(\ell,\bos)$ for $1\le i\le j< r$, be the elements
of $\bu(\lien^\pm[t])$ defined by
 $$\bx^\pm_{i,\overline{j}}(\ell,
\bos)=(x^\pm_{i,\overline{j}}\otimes t^{\bos(1)})\cdots (x^\pm_{i,\overline{j}}\otimes
t^{\bos(\ell)})\qquad\textup{and}\qquad\bx^\pm_{i,j}(\ell,
\bos)=(x^\pm_{i,j}\otimes t^{\bos(1)})\cdots (x^\pm_{i,j}\otimes
t^{\bos(\ell)})$$ if $\ell>0$ and $\bx^\pm_{i,\overline{j}}(0, \emptyset)
=1=\bx^\pm_{i,j}(0, \emptyset)$.  Given $(\ul,\us)\in\bF^j$, set
$$\bx^\pm_{\overline{j}}(\ul, \us) = \bx^\pm_{1,\overline{j}} (\ell_1,
\bos_1)\cdots\bx^\pm_{{j},\overline{j}} (\ell_j, \bos_j)\qquad\textup{and}\qquad\bx_j^\pm(\ul, \us) = \bx^\pm_{1,j} (\ell_1,
\bos_1)\cdots\bx^\pm_{j,j} (\ell_j, \bos_j). $$

Given a POP $\pop^r=(\pattern^r, \{\bos_{i,\overline{j}}\},  \{\bos_{i,j}\} )$, we associate an element $v_{\pop^r}$ of $\mathbf{U}(\lien^-[t])$ as follows:
$$v_{\pop^r}:=\bx^-_{\overline{1}}(\underline{\ell}_{\overline{1}}, \, \underline{\bos}_{\overline{1}})\,
 \bx^-_{1}(\underline{\ell}_{1}, \underline{\bos}_{1})
 \cdots
  \bx^-_{\overline{r-1}}\,(\underline{\ell}_{\overline{r-1}},\, \underline{\bos}_{\overline{r-1}})\,
   \bx^-_{r-1}(\underline{\ell}_{r-1}, \underline{\bos}_{r-1})\,
    \bx^-_{\overline{r}}(\underline{\ell}_{\overline{r}},\, \underline{\bos}_{\overline{r}}),$$
where 
\begin{equation}\label{eq_ljsj}
\underline{\ell}_{\overline{j}}=(\ell_{1, \overline{j}}, \ldots,\ell_{j, \overline{j}}), \quad \underline{\bos}_{\overline{j}}=(\bos_{1, \overline{j}},\ldots,\bos_{j, \overline{j}}),\quad \underline{\ell}_{{j}}=(\ell_{1, {j}}, \ldots,\ell_{j, {j}}), \quad\textup{and}\quad \underline{\bos}_{{j}}=(\bos_{1, {j}},\ldots,\bos_{j, {j}}).
\end{equation}

Given a restricted POP $\pop^{r-1/2}=(\pattern^{r-1/2}, \{\bos_{i,\overline{j}}\},  \{\bos_{i,j}\} )$, we associate an element $v_{\pop^{r-1/2}}$ of $\mathbf{U}(\lien_{r-1/2}^-[t])$ as follows:
$$v_{\pop^{r-1/2}}:=\bx^-_{\overline{1}}(\underline{\ell}_{\overline{1}},\, \underline{\bos}_{\overline{1}})\,
 \bx^-_{1}(\underline{\ell}_{1}, \underline{\bos}_{1})\,
 \cdots
  \bx^-_{\overline{r-1}}\,(\underline{\ell}_{\overline{r-1}}, \,\underline{\bos}_{\overline{r-1}})\,
   \bx^-_{r-1}(\underline{\ell}_{r-1}, \underline{\bos}_{r-1}),$$
   where $\underline{\ell}_{\overline{j}},$ $\underline{\bos}_{\overline{j}},$  $\underline{\ell}_{{j}},$ and $ \underline{\bos}_{{j}}$ are  as in equation~\eqref{eq_ljsj}.

For $\lambda,\eta \in P^+$, let $\mathcal{B}^r(\lambda)$ be the subset of $\mathbf{U}(\lien^-[t])$ consisting elements of the form 
$v_{\pop^r}$, where $\pop^r\in\popset^{r}_{\lambda},$ and let $\mathcal{B}^{r-1/2}(\eta)$ be the subset of $\mathbf{U}(\lien_{r-1/2}^-[t])$ consisting elements of the form 
$v_{\pop^{r-1/2}}$, where $\pop^{r-1/2}\in\popset^{r-1/2}_{\eta}.$
Observe the following:
\begin{equation}\label{eq_basis1}
\mathcal{B}^r(\lambda)=\bigsqcup_{(\ul,\,\us)\in\bF^r(\lambda)}\mathcal{B}^{r-1/2}(\lseq-\ul)\,\bxur^-(\ul,\us),
\end{equation}
\begin{equation}\label{eq_basis2}
\mathcal{B}^{r-1/2}(\eta)=\bigsqcup_{(\ul',\,\us')\in\bF^{r-1}(\eta)}\mathcal{B}^{r-1}(\etaseq-\ul')\,\bxurm^-(\ul',\us').
\end{equation}

\begin{lemma} \label{lem_weight} Let $\lambda\in P^+$. 
For a POP $\pop^r$ with bounding sequence $\lseq$,
the weight of the element $v_{\pop^r}\,w_\lambda$ in $W(\lambda)$ is equal to the weight of  the POP $\pop^r$.
\end{lemma}
\begin{proof}
It is easy to observe that the weight of the element $v_{\pop^r}\,w_\lambda$ in $W(\lambda)$ is equal to $$\lambda-\sum_{1\leq i\leq j<r}\ell_{i,j}\,\alpha_{i,j}-\sum_{1\leq i\leq j\leq r} \ell_{i,\overline{j}}\,\alpha_{i,\overline{j}}.$$
By expressing this as a linear combination of $\varepsilon_1,\ldots,\varepsilon_r$, we get the result. 
\end{proof}
\begin{proposition}\label{prop_count}
\begin{enumerate}
\item\label{prop_card1}
For $\eta=n_1\omega_1+\cdots+n_r\omega_r\in P^+$,  we have 
$$|\mathcal{B}^{r-1/2}({\eta})|=\prod_{i=1}^{r}\left(\binom{2r-1}{i}-\binom{2r-1}{i-2}\right)^{n_i}.$$
\item \label{prop_card2}
 For $\lambda=m_1\omega_1+\cdots+m_r\omega_r\in P^+$, we have
$$|\mathcal{B}^r({\lambda})|=\prod_{i=1}^r\left(\binom{2r}{i}-\binom{2r}{i-2}\right)^{m_i}.$$
\end{enumerate}
\end{proposition}
\begin{proof}
We proceed by induction on $r$. In case $r=1$, it is easy to see that $|\mathcal{B}^{1/2}({\eta})|=0$ and $|\mathcal{B}^1({\lambda})|=\sum_{\ell=0}^{m_1}\binom{m_1}{\ell}=2^{m_1}$. 
From \eqref{eq_basis2}, we have
$$|\mathcal{B}^{r-1/2}({\eta})|=\sum_{\ul'}|\mathcal{B}^{r-1}(\etaseq-\ul')|\prod_{i=1}^{r-1}\binom{n_i}{\ell'_i},$$
where the sum is over all $\ul'=(\ell'_1,\ldots,\ell'_{r-1})\in\bn^{r-1}$ such that $0\leq \ell'_i\leq n_i,$ for $1\leq i<r$. Set $\ell'_{r}=0$. 
Now using the induction hypothesis, we get
 \begin{align*}
 |\cal{B}^{r-1/2}({\eta})| =&
\sum_{\ul'} \prod_{i=1}^{r-1} \left(\binom{2r-2}{i}-\binom{2r-2}{i-2}\right)^{n_i-\ell'_i+\ell'_{i+1}} \left(\binom{2r-2}{r-1}-\binom{2r-2}{r-3}\right)^{n_r}\,\,
\prod_{i=1}^{r-1}\binom{n_i}{\ell'_i} \\=&\prod_{ i=1}^{r-1}
\sum_{\ell'_i=0}^{n_i} \binom{n_i}{\ell'_i}
\left(\binom{2r-2}{i}-\binom{2r-2}{i-2}\right)^{n_i-\ell'_i} \left(\binom{2r-2}{i-1}-\binom{2r-2}{i-3}\right)^{\ell'_i} \\ &\times  \left(\binom{2r-2}{r-1}-\binom{2r-2}{r-3}\right)^{n_r} \\=&
 \prod_{i=1}^{r-1} \left(\binom{2r-1}{i}-\binom{2r-1}{i-2}\right)^{n_i} \left(\binom{2r-2}{r-1}-\binom{2r-2}{r-3}\right)^{n_r},
 \end{align*}
where the last equality is a consequence of the observations that
$$\binom{2r-2}{i}+\binom{2r-2}{i-1}=\binom{2r-1}{i}\qquad\textup{and}\qquad\binom{2r-2}{i-2}+\binom{2r-2}{i-3}=\binom{2r-1}{i-2}.$$
Since $\binom{2r-2}{r-1}-\binom{2r-2}{r-3}=\binom{2r-1}{r}-\binom{2r-1}{r-2}$, we obtain part \eqref{prop_card1}.  

From  \eqref{eq_basis1}, we have
$$|\mathcal{B}^r({\lambda})|=\sum_{\ul}|\mathcal{B}^{r-1/2}(\lseq-\ul)| \prod_{i=1}^r\binom{m_i}{\ell_i},$$
where the sum is over all $\ul=(\ell_1,\ldots,\ell_{r})\in\bn^r$ such that $0\leq \ell_i\leq m_i,$ for $1\leq i\leq r$. Set $\ell_{r+1}=0$. Now using part \eqref{prop_card1}, we get
\begin{align*}
|\mathcal{B}^r(\lambda)|&=\sum_{\ul}\prod_{i=1}^{r}\left(\binom{2r-1}{i}-\binom{2r-1}{i-2}\right)^{m_i-\ell_i+\ell_{i+1}} \prod_{i=1}^r\binom{m_i}{\ell_i}\\
&=\prod_{i=1}^r\sum_{\ell_i=0}^{m_i}\binom{m_i}{\ell_i}\left(\binom{2r-1}{i}-\binom{2r-1}{i-2}\right)^{m_i-\ell_i} \left(\binom{2r-1}{i-1}-\binom{2r-1}{i-3}\right)^{\ell_i}.
\end{align*}
 Now the result follows from the easy observations that 
 $$\binom{2r-1}{i}+\binom{2r-1}{i-1}=\binom{2r}{i}\qquad\textup{and}\qquad \binom{2r-1}{i-2}+\binom{2r-1}{i-3}=\binom{2r}{i-2}.$$
\end{proof}
The proof of the following result is similar to that of \cite[Corollary 2.1.2]{CL}.
\begin{proposition}\label{dimgreater}
For $\lambda\in P^+$, we have $\dim W(\lambda)\geq|\mathcal{B}^r(\lambda)|$. 
\end{proposition}
The following theorem is the main result of this paper.
\begin{theorem}\label{weylbasis}
Let $\lambda\in P^+.$ The set $\{b\,w_\lambda:b\in \mathcal{B}^r(\lambda)\}$ is a basis of $W(\lambda)$.
\end{theorem}
This theorem is proved in \S\ref{sec_mtproof}. We now deduce some corollaries.
\subsubsection{}
First, we give a fermionic formula for the character of $W(\lambda)$.

Given integers $s, n\in\bn $, set
$$[n]_q=\frac {1-q^n}{1-q}, \qquad [n]_q! = [1]_q \cdots
[n]_q, \qquad \bin{n}{s}_q = \frac{[n]_q!}{[s]_q! [n-s]_q!}, \ \ s\leq n, \qquad\bin{n}{s}_q=1, \ \ s\geq n.$$
All these elements are polynomials in $q$ with non-negative integer 
coefficients. 

\begin{corollary}\label{fermionic_char}
For $\lambda=\sum_{i=1}^rm_i\omega_i=\sum_{i=1}^r \lambda_i\,\varepsilon_i\in P^+$, we have
\begin{align*}
 \textup{ch}_q W(\lambda)
=&\sum_{\pop\in\popset^r_\lambda}e^{\textup{weight}(\pop)} q^{|\pop|}\\
=&\sum_{{(\ell_{i,j})\in\bn^{r(r-1)/2}}\atop{1\le i\le j< r}}\sum_{{(\ell_{i,\overline{j}})\in\bn^{r(r+1)/2}}\atop{1\le i\le j\le r}}
e^{\left(\lambda-\sum\limits_{1\leq i\leq j<r}\ell_{i,j}\,\alpha_{i,j}-\sum\limits_{1\leq i\leq j\leq r} \ell_{i,\overline{j}}\,\alpha_{i,\overline{j}}\right)}\\
&\prod_{1\le i \le j < r} \left[{ m_i +
\sum\limits_{k=j+1}^{r-1} (\ell_{i+1,k}-\ell_{i,k}) +\sum\limits_{k=j+1}^r
(\ell_{i+1,\overline{k}}-\ell_{i,{\overline{k}}})\atop{\ell_{i,j}}} \right]_q\\
&\prod_{1\le i < j \le r} \left[{ m_i +
\sum\limits_{k=j}^{r-1} (\ell_{i+1,k}-\ell_{i,k}) +\sum\limits_{k=j+1}^r
(\ell_{i+1,\overline{k}}-\ell_{i,{\overline{k}}})\atop{\ell_{i,\overline{j}}}} \right]_q
\prod_{i=1}^{r} \left[{ \lambda_i -
\sum\limits_{k=i}^{r-1} \ell_{i,k} -\sum\limits_{k=i+1}^r
\ell_{i,{\overline{k}}}\atop{\ell_{i,\overline{i}}}} \right]_q
\end{align*}
\end{corollary}
\begin{proof}
The proof follows from Theorem \ref{weylbasis}, Lemma \ref{lem_weight}, and the observation that for a fixed pair of integers  $\ell, m \in \bn$  we
have $\sum_{0\le\bos(1)\leq\cdots\leq\bos(\ell)\leq m-\ell} q^{\bos(1) + \cdots +\bos(\ell)} = \bin{m}{\ell}_q.$
\end{proof}

The following  result is immediate from Theorem~\ref{weylbasis} and the observation that the zeroth graded piece of $W(\lambda)$ is isomorphic to $V(\lambda)$.
\begin{corollary}\label{irrbasis}
For $\lambda\in P^+$,  let $v_\lambda$ denote a highest weight vector of $V(\lambda)$. 
The elements $$(x^-_{1,\overline{1}})^{\ell_{1,\overline{1}}}\,(x^-_{1,1})^{\ell_{1,1}}\,(x^-_{1,\overline{2}})^{\ell_{1,\overline{2}}}\,(x^-_{2,\overline{2}})^{\ell_{2,\overline{2}}}\,(x^-_{1,2})^{\ell_{1,2}}\,(x^-_{2,2})^{\ell_{2,2}}\cdots(x^-_{r,\overline{r}})^{\ell_{r,\overline{r}}}\,v_\lambda,$$ where $\ell_{i,\overline{j}},\,\ell_{i,j}\in\bN$ are  the differences corresponding to a pattern with bounding sequence $\lseq$, are a basis for $V(\lambda)$.
\end{corollary}

The next result follows from  Theorem~\ref{weylbasis} by using Proposition~\ref{prop_count}.
\begin{corollary}\label{dimformula}
For $\lambda=m_1\omega_1+\cdots+m_r\omega_r\in P^+$, we have
$$\dim W(\lambda)=\prod_{i=1}^r \left(\dim W(\omega_i)\right)^{m_i}.$$
\end{corollary}
\begin{remark}
The above result is proved in \cite{N}. The methods they use are quite different from ours.
\end{remark}

\subsection{}
 Let $\bi^m=\bn^m\times \bn^m$. We recall from \cite[\S 2.2.2]{CL}  the total  order on
$\bi^m = \{(\ul, \ud)\}$:
\begin{eqnarray*}
\ul \eql \um & \mbox{if} & \ell_1 = m_1,\dots, \ell_{s-1} =
m_{s-1},\ \ell_s < m_s,\\ \ud \eqd \ue & \mbox{if} & d_m=
e_m,\dots, d_{s+1} = e_{s+1},\ d_s
> e_s\end{eqnarray*} for some $1\le s\le m$. Finally $$
(\ul,\ud)> (\um, \ue) \ \ \mbox{if} \ \  \ul \eql \um \ \
\mbox{or} \  \  \ul = \um, \ \  \ud \eqd \ue.$$
We now define a total order on $\bi^{m-1}\times\bi^{m}$:
$$
(\ul',\ud',\ul,\ud)> (\um',\ue',\um, \ue) \ \ \mbox{if} \ \  (\ul,\ud) > (\um,\ue) \ \
\mbox{or} \  \ (\ul,\ud) = (\um,\ue), \ \ (\ul',\ud')>(\um',\ue') .$$
For $(\ell, \bos) \in \bF$, set
$|\bos|=\sum_{p=1}^{\ell}\bos(p)$, if $\ell>0$ and
$|\emptyset|=0$.
Given $(\ul, \us) \in \bF^m$ let $|\us| = (|\bos_1|, \dots,
|\bos_m|)$. Note that $(\ul,|\us|)\in\bi^m$.

\medskip
 Given $(\boi', \boi)\in\bi^{r-1}\times \bi^{r}$ and $\lambda\in P^+$, define $\lieg_{r-1}[t]$--modules,
\begin{eqnarray*}&{W}(\lambda)^{\ge (\boi', \boi)}
&=\sum_{\substack{{(\ul',\,\us',\,\ul,\,\us)\in\bF^{r-1}\times\bF^{r}}\\ {(\ul',\,|\us'|,\,\ul,\,|\us|)\ge (\boi', \boi)}}}\bu(\lieg_{r-1}[t])\,\bxurm^-(\ul',\us')\,\bxur^-(\ul,\us)\,w_\lambda,\\ & W(\lambda)^{> (\boi', \boi)}
&=\sum_{\substack{{(\ul',\,\us',\,\ul,\,\us)\in\bF^{r-1}\times\bF^{r}}\\ {(\ul',\,|\us'|,\,\ul,\,|\us|)> (\boi', \boi)}}}\bu(\lieg_{r-1}[t])\,\bxurm^-(\ul',\us')\,\bxur^-(\ul,\us)\,w_\lambda ,\\ &
\pgr(W(\lambda))&=\bigoplus_{(\boi', \boi)\in\bi^{r-1}\times \bi^{r}}W(\lambda)^{\ge(\boi', \boi)}/W(\lambda)^{>(\boi', \boi)}.\end{eqnarray*}
Let $$\pgr^{(\boi', \boi)}:
W(\lambda)^{\ge(\boi', \boi)}\to W(\lambda)^{\ge(\boi', \boi)}/W(\lambda)^{>(\boi', \boi)}$$
be the canonical projection, which is clearly a map of $\lieg_{r-1}[t]$--modules. 

We shall deduce Theorem~\ref{weylbasis} from the next proposition
\begin{proposition}\label{assgraded}
  Let $\lambda\in P^+$ and $r\ge
2$.
\begin{enumerate}
\item  \label{assgradedp1}For  $(\ul,\us)\in\bF^r(\lambda)$ and $(\ul',\us')\in \bF^{r-1}(\lseq-\ul)$,
 there exists a map of $\lieg_{r-1}[t]$--modules
$$\psi^{\,\ul',\us',\ul,\us}: W({\lseq-\ul-\ul'})\to\pgr(W(\lambda))$$ given
by extending  $$\psi^{\,\ul',\us',\ul,\us}(w_{\lseq-\ul-\ul'})
=\pgr^{(\ul', |\us'|,\ul, |\us|)} (\bxurm^-(\ul',\us')\,\bxur^-(\ul,\us)\,w_\lambda).$$ 
\item \label{assgradedp2} The images of $\psi^{\,\ul',\us',\ul,\us}$  for
 $(\ul,\us)\in\bF^r(\lambda)$  and $(\ul',\us')\in \bF^{r-1}(\lseq-\ul)$
span $\pgr(W(\lambda))$.
\end{enumerate}
\end{proposition}
\subsection{}\label{sec_mtproof}{\em Proof of Theorem~\ref{weylbasis} from Proposition~\ref{assgraded}}.
We proceed by induction on $r$. For $r=1$, this holds by the results in \cite[ \S6]{CP}.  Let $r\geq 2$, and  assume the
result for $r-1$. Using Proposition~\ref{assgraded} and the
induction hypothesis we see that $
\pgr(W(\lambda))$ is spanned by
the  sets
$$\pgr^{(\ul', |\us'|,\ul, |\us|)}(\cal{B}^{r-1}(\lseq-\ul-\ul')\,\bxurm^-(\ul',\us')\,\bxur^-(\ul,\us) \,w_\lambda), \quad\textup{where}\quad (\ul,\us)\in\bF^{r}(\lambda), \, (\ul', \us')\in\bF^{r-1}(\lseq-\ul).$$
Since $W(\lambda)$ is finite-dimensional,  we get analogous to the proof of \cite[Proposition~2.2.2]{CL} that  the spanning set of $\pgr(W(\lambda))$ gives the spanning set of $W(\lambda)$.  
Using
equations~\eqref{eq_basis1}-\eqref{eq_basis2}, we get  that
$\cal{B}^{r}(\lambda)\,w_\lambda$ spans $W(\lambda)$ and hence $\dim W(\lambda)\leq|\cal{B}^{r}(\lambda)|$. 
Since Proposition~\ref{dimgreater} gives the reverse inequality, it follows that $\dim W(\lambda)=|\cal{B}^{r}(\lambda)|$. Hence $\cal{B}^{r}(\lambda)\,w_\lambda$ is a basis of $W(\lambda)$ and the theorem is proved.

\subsection{} We conclude this section with a final corollary of
Theorem~\ref{weylbasis} and Proposition~\ref{assgraded}.

\begin{corollary}\label{weylfiltration}
 For $\lambda = \sum_{i=1}^r
m_i \omega_i\in P^+$, we have an isomorphism of $\lieg_{r-1}[t]$--modules
 $$\pgr(W(\lambda)) \cong\bigoplus_{\substack{{(\ul,\,\us)\in\bF^r(\lambda)}\\{(\ul',\,\us')\in\bF^{r-1}(\lseq-\ul)}}}W({\lseq-\ul-\ul'})
 =\bigoplus_{\substack{{\ul\in\bn^r,\, \ul'\in\bn^{r-1}}\\{\ell_i \le m_i}, \,\ell'_i\le m_i-\ell_i+\ell_{i+1}}}
m_{\ul',\,\ul,\,\lambda}\, m_{\ul,\,\lambda}\,W({\lseq-\ul-\ul'}), $$  {where}
$$ m_{\ul,\, \lambda}=\prod_{i=1}^r \binom{m_i}{\ell_i}, \qquad m_{\ul',\, \ul,\, \lambda}=\prod_{i=1}^{r-1} \binom{m_i-\ell_i+\ell_{i+1}}{\ell'_i}.  $$  In particular, the module $W(\lambda)$
admits a filtration by $\lieg_{r-1}[t]$--modules such that the
successive quotients are local Weyl modules for $\lieg_{r-1}[t]$.
\end{corollary}

\section{Proof of Proposition~\ref{assgraded}}
\subsection{}
For $(\boi', \boi)\in\bi^{r-1}\times \bi^{r}$, let
\begin{eqnarray*}&{\mathbf{U}(\lien^-[t]))^{\ge (\boi', \boi)}}
&=\sum_{\substack{{(\ul',\,\us',\,\ul,\,\us)\in\bF^{r-1}\times\bF^{r}}\\ {(\ul',\,|\us'|,\,\ul,\,|\us|)\ge (\boi', \boi)}}}\bu(\lien^-_{r-1}[t])\,\bxurm^-(\ul',\us')\,\bxur^-(\ul,\us),\\ & \mathbf{U}(\lien^-[t])^{> (\boi', \boi)}
&=\sum_{\substack{{(\ul',\,\us',\,\ul,\,\us)\in\bF^{r-1}\times\bF^{r}}\\ {(\ul',\,|\us'|,\,\ul,\,|\us|)> (\boi', \boi)}}}\bu(\lien^-_{r-1}[t])\,\bxurm^-(\ul',\us')\,\bxur^-(\ul,\us).
\end{eqnarray*}

Clearly we have \begin{align*}
&W(\lambda)^{\ge (\boi',\boi)} \supseteq \bu(\lien^-[t])^{\ge (\boi',\boi)} \,w_\lambda,\quad 
&W(\lambda)^{> (\boi',\boi)}\supseteq\bu(\lien^-[t])^{> (\boi',\boi)} \,w_\lambda,\qquad
&\forall\,\,(\boi',\boi)\in\bi^{r-1}\times \bi^{r}.
\end{align*}
 We will show in
Proposition~\ref{act2} that  the reverse inclusions hold as well.

Let $\mathbf{e}_i\in\bN^r, i\in I,$ denote the standard basis vectors.
\begin{lemma}\label{lem1act}
 Let $(\ul,\us)\in \bF^r$ and $\lambda\in P^+$.
\begin{enumerate}
\item\label{lem_act1} For $g=h\otimes t^{s+1}, \,x_{i,{j}}^+\otimes t^s, \,x_{i,\overline{j}}^+\otimes t^s$, where $1\le i\le j<r$, $h\in\lieh$, and $s\in \bN$, the element $g\,\bxur^-(\ul,\us)\, w_\lambda\in W(\lambda)$
 is a linear combination of elements of the form
$\bxur^-(\um, \up)\,w_\lambda$ with $(\um, |\up|) > (\ul, |\us|)$.

\item\label{lem_act2} For $1\le i\le r$ and $s\in\bN$, the element $(x^+_{i,r}\otimes t^s)\,\bxur^-(\ul,\us)\,w_\lambda\in W(\lambda)$
 is a linear combination of elements of the form
$$\bxur^-(\um, \up)\,w_\lambda, \quad 
(x^-_{k,i-1}\otimes t^{s'})\,\bxur^-(\um, \up)\,w_\lambda, \quad\textup{and}\quad
\delta_{i,r}\,(x^-_{m,\overline{k}}\otimes t^{s'})\, \bxur^-(\um, \up)\,w_\lambda,$$ with $\um\eql \ul,\, m\le k<i$.
\item\label{lem_act3} For all $(\un, \uq)\in \bF^r$ with $\un\neq\underline{0},$  we have
$$\bxur^+(\un, \uq)\, \bxur^-(\ul,\us)\,w_\lambda \in
\sum\limits_{\{(\um,\,\up)\in\bF^r: \um\eql\ul\}}
\mathbf{U}(\lien^-_{r-1/2}[t])\,\bxur^-(\um,\up)\,w_\lambda.$$
\end{enumerate}
\end{lemma}
\begin{proof}
 First, note that $g\,\bxur^-(\ul,\us)\,w_\lambda=[g,\, \bxur^-(\ul,\us)] w_\lambda$, for all $g\in\lien^+[t]\oplus\lieh t[t]$. The proof of 
 part~\eqref{lem_act1} in the cases
$g=h\otimes t^{s+1}, \,x_{i,{j}}^+\otimes t^s$ is same as that of \cite[Lemma 3.2.1]{CL}. We now consider the case $g=x_{i,\overline{j}}^+\otimes t^s$, where $1\le i\le j<r$ and $s\in\bn$. For $1\le k\le k'\le r$, we have 
$$ 
[x_{i,\overline{j}}^+, \,x_{k,r}^-]=\delta_{k,i}\, x^+_{j,r-1}+\delta_{k,j}\, x^+_{i,r-1}  \qquad\textup{and}\qquad [x^+_{p,r-1}, x^-_{k',r}]=-2\delta_{k',p}\,x_{r,r}^-,\,\,\forall\,\, p\in\{i,j\}.$$ 
This implies that $(x_{i,\overline{j}}^+\otimes t^s)\,\bxur^-(\ul,\us)\,w_\lambda\in W(\lambda)$
 is a linear combination of elements of the form
$\bxur^-(\um, \up)\,w_\lambda$ with $\um=\ul-\mathbf{e}_i-\mathbf{e}_j+\mathbf{e}_r$.
Clearly $\um\eql\ul$. This completes the proof of part~\eqref{lem_act1}.

The proof of part~\eqref{lem_act2} follows from the following observations:
$$[x_{i,r}^+,\, x^-_{k,r}]=\begin{cases}
x_{i, k-1}^+, &i<k,\\
\alpha^\vee_{i,r}, &i=k,\\
x^-_{k,i-1}, &i>k.
\end{cases}, \qquad [x_{i,k-1}^+,\, x_{m,r}^-]=0, \  m\geq k, \qquad [x^-_{m,r}, \,x^-_{k,i-1}]=\delta_{i,r} \,x^-_{m,\overline{k}}, \  m\leq k.$$

To prove part~\eqref{lem_act3}, we write 
$\bxur^+(\un, \uq)$ in the order  $\bx^+_{r,r} (n_r, \boq_r)\,
\cdots\bx^+_{1,r} (n_1, \boq_1)$. 
Since 
$$[x^+_{p,r},\,x^-_{k, i-1}]=0,\,\,\forall\,\,k<i\leq p\leq r \qquad\textup{and}\qquad [x^+_{r,r},\,x^-_{k,i-1}]=0=[x_{r,r}^+, \,x_{m,\overline{k}}^-], \,\,\forall \,\,m\le k<i\leq r,$$  
the proof  follows by repeatedly using part~\eqref{lem_act2}.
\end{proof}

\begin{lemma}\label{lem2act}
 Let $(\ul',\us')\in \bF^{r-1}$.
 \begin{enumerate}
 \item\label{lem2_act1}  For $g=h\otimes t^{s+1},\,x_{i,{j}}^+\otimes t^s$, where $1\leq i\leq j<r-1$, $h\in\lieh_{r-1}$, and $s\in \bN,$ the element $[g,\,\bxurm^-(\ul',\us')]$
 is a linear combination of elements of the form
$\bxurm^-(\um',\up')$ with $(\um', |\up'|) > (\ul', |\us'|)$.
\item\label{lem2_act2} For $g=x_{i,\overline{j}}^+\otimes t^s$, where $1\le i\le j<r$ and $s\in \bN$, 
the element $[g,\,\bxurm^-(\ul',\us')]$
 is a linear combination of elements of the form 
 $\bxurm^-(\um',\up')\,(x^+_{k,r}\otimes t^{s'})$ with $\um'\eql\ul'$, $k\in\{i, j, r\}.$
 \end{enumerate}
\end{lemma}
\begin{proof}
The proof of part~\eqref{lem2_act1} is similar to that of \cite[Lemma 3.2.1]{CL}. The proof of  part~\eqref{lem2_act2} follows from the observations that
$$ 
[x_{i,\overline{j}}^+, \,x_{k,r-1}^-]=-\delta_{k,i} \,x^+_{j,r}-\delta_{k,j}\, x^+_{i,r}, \qquad [x^+_{p,r}, \,x^-_{k',r-1}]=-2\delta_{k', p}\,x_{r,r}^+,  \qquad\textup{and}\qquad[x^+_{r,r}, \,x^-_{k'',r-1}]=0,$$ 
for $1\le i\le j<r$,  $1\le k\le k'\le k''< r$, and $p\in\{i,j\}$.
\end{proof}

\begin{proposition}\label{act2} 
For $\lambda\in P^+,$ we have 
$$W(\lambda)^{\ge (\boi', \boi)} =
\bu(\lien^-[t])^{\ge (\boi', \boi)} \,w_\lambda, \quad W(\lambda)^{> (\boi', \boi)} =
\bu(\lien^-[t])^{> (\boi', \boi)} \,w_\lambda, \quad\forall\,\, (\boi', \boi)\in \bi^{r-1}\times \bi^{r}.$$
\end{proposition}
\begin{proof}
It is enough to show for $(\ul',\us')\in\bF^{r-1}$ and $(\ul,\us)\in\bF^r$ that $$
\bu(\lieg_{r-1} [t])\,  \bxurm^-(\ul',\us')\,\bxur^-(\ul,\us)\,
w_\lambda\subset \bu(\lien^-[t])^{\ge (\ul',\,|\us'|,\,\ul,\,|\us|)}\,
w_\lambda.$$ Since $\bu(\lieg_{r-1} [t]) =\bu(\lien^-_{r-1}
[t])\bu(\lieb_{r-1}^+ [t]),$ it suffices to prove that
 $$\bu(\lieb_{r-1}^+ [t])\,
\bxurm^-(\ul',\us')\,\bxur^-(\ul,\us) \,w_\lambda\subset\bu(\lien^-[t])^{\ge (\ul',\,|\us'|,\,\ul,\,|\us|)}\,
w_\lambda.$$
For $g \in \lieb ^+_{r-1}[t]$, we have 
\begin{equation}\label{eq_bactonxrmxr}
g\cdot\bxurm^-(\ul',\us')\,\bxur^-(\ul,\us) \,w_\lambda =   \bxurm^-(\ul',\us')\, g\, \bxur^-(\ul,\us) \,w_\lambda+[g,\,\bxurm^-(\ul',\us')]\,\bxur^-(\ul,\us) \,w_\lambda.
\end{equation}
For $h\in\lieh_{r-1}$, we observe that
\begin{equation}\label{eqhact}
h\cdot \bxurm^-(\ul',\us')\,\bxur^-(\ul,\us) \,w_\lambda=
\langle \lseq-\ul-\ul', \,h\rangle
\bxurm^-(\ul',\us')\,\bxur^-(\ul,\us) \,w_\lambda.
\end{equation}

Consider the subalgebras $\mathfrak{a}$ and $\lieb'$  of $\lieb^+_{r-1}$ given by 
$$\mathfrak{a}=\bigoplus\limits_{1\leq i\leq j<r}\complex\, x^+_{i,\overline{j}}\qquad\textup{and}\qquad\lieb'= \lieh_{r-1}\bigoplus\limits_{1\leq i\leq j<r-1} \complex \,x^+_{i,j}.$$ 
Observe that $\lieb^+_{r-1}=\mathfrak{a}\oplus\lieb'$ and hence $\bu(\lieb^+_{r-1}[t])=\bu(\mathfrak{a}[t])\,\bu(\lieb'[t])$.  
Using equations~\eqref{eq_bactonxrmxr}--\eqref{eqhact}, Lemmas~\ref{lem1act}~\eqref{lem_act1}
and \ref{lem2act}~\eqref{lem2_act1}, we get that 
$\bu(\lieb' [t])\,  \bxurm^-(\ul',\us')\,\bxur^-(\ul,\us)\,
w_\lambda$ is contained in the span of the elements of the form 
$\bxurm^-(\um',\up')\,\bxur^-(\um,\up)\,
w_\lambda$ with $(\um',\,|\up'|,\,\um,\,|\up|)\geq(\ul',\,|\us'|,\,\ul,\,|\us|)$.
Hence we only need to show for $g_1, g_2, \ldots, g_m\in\{x^+_{i, \overline{j}}\otimes t^s: s\in\bN, \,1\leq i\leq j<r\}$ that 
\begin{align}\label{aact} g_1\,g_2\cdots g_m \, \bxurm^-(\ul',\us')\,\bxur^-(\ul,\us)\,
w_\lambda\in \sum_{\{(\um,\,\up)\in\bF^r:(\um,\,|\up|)>(\ul,\,|\us|)\}} \bu(\lien_{r-1/2}^-[t])\,\bxur^-(\um, \up)\,
w_\lambda,\end{align}
Since $[x^+_{i,\overline{j}}, \,x^+_{k,r}]=0$, for all $1\leq i\leq j<r, \,1\leq k\leq r$, by using repeatedly equation~\eqref{eq_bactonxrmxr}, Lemmas~\ref{lem1act}~\eqref{lem_act1} and \ref{lem2act}~\eqref{lem2_act2}, we get that the element $g_1\,g_2\cdots g_m \, \bxurm^-(\ul',\us')\,\bxur^-(\ul,\us)\,
w_\lambda$ is in the span of the elements from 
$$\bxurm^-(\ul',\us')\,\bxur^-(\um,\up)\,
w_\lambda,\qquad (\um, |\up|)>(\ul, |\us|)$$ together with the elements from 
$$
\bxurm^-(\um',\up')\,\bxur^+(\un,\uq)\,\bxur^-(\um,\up)\,
w_\lambda,\qquad\um'\eql \ul',\quad (\un,\uq)\in F^r, \,\, \un\neq\underline{0}, \quad (\um, |\up|)\geq (\ul, |\us|).$$
Now, the proof of \eqref{aact} follows from Lemma~\ref{lem1act}~\eqref{lem_act3}.
\end{proof}
\subsubsection{}{\em Proof of Proposition~\ref{assgraded}~\eqref{assgradedp1}.} 
Using equations~\eqref{eq_bactonxrmxr}--\eqref{eqhact} and Lemmas~\ref{lem1act}--\ref{lem2act},  we get that the
element $$\pgr^{(\ul,
|\us|,\ul',|\us'|)} (\bxurm^-(\ul',\us')\,\bxur^-(\ul,\us) \,w_\lambda)\in\pgr(W(\lambda))$$ satisfies the
relations in~\eqref{weylthmreln1} with the weight $ \lseq-\ul-\ul'$. 
Theorem~\ref{theorem_uni1} now implies the existence of the map
$\psi^{\,\ul',\us',\ul,\us}$.
\subsubsection{} 
The proof of Proposition~\ref{assgraded}~\eqref{assgradedp2} is immediate from the following  proposition by  using Proposition~\ref{act2}.
\begin{proposition}\label{mainprop}
Given $\lambda\in P^+,$ we have the following:
\begin{enumerate}
\item\label{mainpropp1} For all $(\ul,\us)\in\bF^r$, we have
\begin{align*}
\bxur^-(\ul,\us)\,w_\lambda  \in &
\sum_{\{\up: (\ul,\,\up)\in\bF^r(\lambda),\, |\up|=|\us|\}} \complex\,\bxur^-(\ul,\up)\,w_\lambda\\ &+\sum_{\{(\um,\,\up)\in\bF^r:(\um,\,|\up|)>(\ul,\,|\us|)\}} \bu(\lien^-_{r-1/2}[t])\,\bxur^-(\um,\up)\,w_\lambda.
\end{align*}
\item\label{mainpropp2}
Let $(\ul,\us)\in\bF^r(\lambda)$. For all $(\ul',\us')\in\bF^{r-1},$ we have
\begin{align*}
\bxurm^-(\ul',\us')\,\bxur^-(\ul,\us) \,w_\lambda\in &\sum_{\{\up ': (\ul',\, \up')\in\bF^{r-1}(\lseq-\ul),\,|\up'|=|\us'|\}} \complex\,\bxurm^-(\ul', \up')\,\bxur^-(\ul,\us) \,w_\lambda \\ & +
\sum_{\{(\um',\,\up')\in\bF^{r-1} : (\um',\, |\up'|)>(\ul',\, |\us'|)} \bu(\lien^-_{r-1}[t]) \,\bxurm^-(\um', \up')\,\bxur^-(\ul,\us) \,w_\lambda \\ & +
\sum_{\{(\um,\,\up)\in\bF^r : (\um,\,|\up|)>(\ul,\,|\us|)\}} \bu(\lien^-_{r-1/2}[t])\,\bxur^-(\um,\up)\,w_\lambda.
\end{align*}
\end{enumerate}
\end{proposition}
\subsubsection{} We first prove Proposition~\ref{mainprop}~\eqref{mainpropp2}.
Let $\lien'^\pm$ and $\lieg'$ be the subalgebras of $\lieg$ given by
$$
\lien'^\pm=\bigoplus_{1\leq i\leq j<r} \complex\, x_{i,j}^\pm, \qquad  \lieg'=\lien'^-\oplus\lieh_{r-1}\oplus\lien'^+.$$  Observe that 
 $\lieg'$  is isomorphic to the special linear Lie algebra $\mathfrak{sl}_r$ of rank $r-1$. 
 
Given $\boi\in\bi^r$ and $\lambda\in P^+$, define $\lieg'[t]$--modules,
\begin{align*}{\pW}(\lambda)^{\ge \boi}
=\sum_{\{(\ul,\,\us)\in\bF^r: (\ul,\,|\us|)\ge \boi\}}\bu(\lieg'[t])\,\bxur^-(\ul,\us)\,w_\lambda,\ \ \pW(\lambda)^{> \boi}
=\sum_{\{(\ul,\,\us)\in\bF^r: (\ul,\,|\us|)> \boi\}}\bu(\lieg'[t])\,\bxur^-(\ul,\us)\,w_\lambda.
\end{align*}
Let $\gr^\boi:
\pW(\lambda)^{\ge\boi}\to \pW(\lambda)^{\ge\boi}/\pW(\lambda)^{>\boi}$
be the canonical projection, which is clearly a map of $\lieg'[t]$--modules. 
Using Lemma~\ref{lem1act}~\eqref{lem_act1}, observe that
\begin{align}\label{WGvsNeq}
\pW(\lambda)^{\ge \boi} =\sum_{\{(\ul,\,\us)\in\bF^r:
(\ul,\,|\us|)\ge \boi\}}\bu(\lien'^-[t])\,\bxur^-(\ul,\us) \,w_\lambda,\ \  
\pW(\lambda)^{> \boi} =\sum_{\{(\ul,\,\us)\in\bF^r:
{(\ul,\,|\us|)> \boi}\}}\bu(\lien'^-[t])\,\bxur^-(\ul,\us) \,w_\lambda.
\end{align}

Let $(\ul,\us)\in\bF^r(\lambda)$ and let $\pW(\lseq-\ul)$ be the local Weyl module of $\mathfrak{sl}_r[t]$ with the highest weight $\lseq-\ul$. Using Lemma~\ref{lem1act}~\eqref{lem_act1},  it is easy to see that there exists a $\lieg't]$--modules map from $\pW(\lseq-\ul)$ onto 
$\bu(\lieg'[t])\,\gr^{(\ul,|\us|)} (\bxur^-(\ul,\us)\,w_\lambda)$.
Now the proof of Proposition~\ref{mainprop}~\eqref{mainpropp2} follows from
the proof of \cite[Proposition~2.2.3~(ii)]{CL} by using equation \eqref{WGvsNeq}.

\medskip

The rest of the paper is devoted to proving Proposition~\ref{mainprop}~\eqref{mainpropp1}.
\subsection{} 
We first state some elementary facts about the module $V(\lambda)$.
\begin{lemma}\label{lem_fd} Let $\lambda = \sum_{i=1}^r m_i \omega_i=\sum_{i=1}^r \lambda_i\,\varepsilon_i\in P^+$. Then, we have
\begin{equation*} V(\lambda)=\sum_{\substack{{k_1, \dots, k_r \in \bn}\\{k_1 \le m_1}}}
\bu(\lien^-_{r-1/2})\,(x^-_{1,r})^{k_1}\cdots
(x^-_{r,r})^{k_r}\,v_\lambda.\end{equation*}
\end{lemma}
\begin{proof}
 
 Using Theorem~\ref{theoremofS}~\eqref{theoremofSp1} we see that  $V(\lambda)$ is
spanned by  the sets $\bu(\lien^-_{r-1/2})\,\pv_\eta$, where
$\underline{\eta}=(\eta_1\ge\cdots \ge \eta_r)$ is a sequence of non--negative integers such that
$\lambda_i\ge \eta_i \ge \lambda_{i+1}$, $1\le i \le r$,
and  $\pv_\eta\in V(\lambda)$ satisfies
\begin{equation*} \lien^+_{r-1/2}\,\pv_\eta=0, \qquad
h\,\pv_\eta=\langle\eta, \,h\rangle\pv_\eta, \,\,\forall\,\,h\in\lieh_{r-1}.
\end{equation*}

 Writing $\pv_\eta$ as a $\bu(\lien^-_{r-1/2})$--linear combination of
elements $(x^-_{1,r})^{k_1}\cdots
(x^-_{r,r})^{k_r}\,v_{\lambda}$, $k_1,\cdots ,k_r\in \bn,$ and then by
comparing the coefficients of $\varepsilon_1$ in their weights, we obtain $\eta_1=\lambda_1-k_1-k$ for some $k\geq 0$. Since $\lambda_1\ge\eta_1\ge\lambda_2$, we get $k_1\le m_1$.
\end{proof}

For $1\le m\le n\le r$, we define the subalgebras $\lien_{m,n}^\pm, \lieh_{m,n}$, and $\lien_{m,n-1/2}^\pm$ of $\lieg$ as follows:
 \begin{gather*}
 \lien_{m,n}^\pm=\bigoplus_{m\leq i \leq j < n} \complex \,x_{i,j}^{\pm}\bigoplus_{m\leq i \leq j \leq n} \complex \,x_{i,\overline{j}}^{\pm},
  \quad\qquad\lieh_{m,n}=\bigoplus_{m\le i\le n}\complex\, \alpha_{i,i}^\vee,\\
 \lien_{m,n-1/2}^+=\bigoplus_{m\leq i \leq j < n-1} \complex\, x_{i,j}^{+}\bigoplus_{m\leq i \leq j \leq n} \complex\, x_{i,\overline{j}}^{+}, \qquad\quad \mathfrak{n}^-_{m, n-1/2}=\bigoplus_{m\leq i \leq j < n} \left(\complex\, x_{i,j}^{-}\oplus\complex\, x_{i,\overline{j}}^{-}\right).
\end{gather*}
For $1\le i\le r$, set $\lien_i^\pm=\lien_{1,i}^\pm$ and $\lien_{i-1/2}^\pm=\lien_{1,i-1/2}^\pm$.
Set
$\lieg_{m,n} = \lien_{m,n}^- \oplus \lieh_{m,n} \oplus \lien_{m,n}^+$ and $\lieg_{m,n-1/2} = \lien_{m,n-1/2}^- \oplus \lieh_{m,n-1} \oplus \lien_{m,n-1/2}^+$.
Clearly $\lieg_{m,n}$ is isomorphic to $\mathfrak{sp}_{2(n-m+1)}$  and $\lieg_{m,n-1/2}$ is isomorphic to ${\mathfrak{sp}}_{2(n-m)+1}$.

We can now prove the following stronger statement.
\begin{proposition}\label{prop_fd}
Let $\lambda = \sum_{i=1}^r m_i \omega_i\in P^+$. Then, we have
\begin{equation} V(\lambda)=\sum_{\{k_i: 0\le k_i\le m_i, 1\le i\le r\}}
\bu(\lien^-_{r-1/2})\,(x^-_{1,r})^{k_1}\cdots
(x^-_{r,r})^{k_r}\,v_\lambda.\end{equation}
\end{proposition}

\begin{proof}
We proceed by induction on $r$. For $r=1$, the
result follows from the defining relations of $V(\lambda)$. Note
that $\bu(\lieg_{2,r})\,v_\lambda\cong V(\lambda|_{\lieh_{2,r}})$
as $ \lieg_{2,r}$--modules, and hence we have by the induction
hypothesis
\begin{equation*}
\bu(\lien^-_{2,r})\,v_\lambda=\sum_{\{k_i: 0\le k_i\le m_i,  2\le
i\le r\}} \bu(\lien^-_{2,r-1/2})\, (x^-_{2,r})^{k_2}\cdots
(x^-_{r,r})^{k_r}\,v_\lambda.
\end{equation*}
Combining this  with Lemma~\ref{lem_fd}, we find that
$$V(\lambda)=\sum_{\{k_i: 0\le k_i\le m_i, 1\le i\le r\}} \bu(\lien^-_{r-1/2})\, (x^-_{1,r})^{k_1} \bu(\lien^-_{2,r-1/2})\,
(x^-_{2,r})^{k_2}\cdots (x^-_{r,r})^{k_r}\,v_\lambda.$$  
Since $\lien^-_{2,r-1/2} \subset \lien^-_{r-1/2}$ and $[x^-_{1,r},\,\lien^-_{r-1/2}]\subset\lien^-_{r-1/2}$, the proposition follows.
\end{proof}

\subsection{}  We now recall the projection map $\bpr$ from \cite[\S 3.1.1]{CL}. 
Let
$\bpr: \bu(\lieg[t])\to\bu(\lien^-[t])$ be the projection
corresponding to the vector space decomposition $$\bu(\lieg[t])=\bu(\lien^-[t])\oplus \bu(\lieg[t])(\lieb^+[t]),$$ given
by the Poincare--Birkhoff--Witt theorem. Clearly $\bpr$ is a
$\bn$--graded linear map. 

The next result is immediate from the definition.
\begin{proposition}\label{proppr} 
\begin{enumerate}
\item \label{proppr3} For all $g_1, \, g_2 \in  \bu(\lieg[t])$, we have $\bpr(g_1\, g_2) = \bpr(g_1\,
\bpr(g_2))$.
\item\label{prl} Let $\lambda\in P^+$. Then for all
$x \in \bu(\lien^-[t] \oplus \lieb^+ t[t])$, we have $x\,w_\lambda
= \bpr(x) \,w_\lambda$ in $W(\lambda)$.\end{enumerate}
\end{proposition}

\subsubsection{} Let $\bF_+$ denote the subset of $\bF$ consisting of
pairs $(\ell,\bos)\in\bF$ such that  $\bos(i)>0$ for all $1 \le i \le
\ell$, together with   the pair $(0,\emptyset)$. 

\begin{lemma}\label{lem_dif}  Let $(\ul^+, \us^+)\in \bF^r_+$ with $\ul^+\neq\underline{0}$ and set $i_0:= \textup{max} 
\{i: \ell^+_i\neq0\}.$
For all $(\ul,\us)\in
\bF^r$ with $\ell_i\ge \ell_i^+$, $1\le i\le r$, we have
\begin{align*}\label{eq_dif}
&\bpr\left(\bxur^+ (\ul^+, \us^+) \,\bxur^- (\ul, \us)\right) -
\prod_{i=1}^r \bpr \left( \bx^+_{i,r} (\ell^+_i, \bos^+_i)\,
\bx^-_{i,r} (\ell_i, \bos_i)\right)\\& \in \sum_{(\um,\,|\up|)>(\ul-\ul^+, \,|\us| + |\us^+|)}
\bu(\lien_{i_0-1/2}^-[t])\,\bxur^- (\um,\up)
\end{align*}
In particular, $$\bpr\left(\bxur^+ (\ul^+, \us^+)\, \bxur^- (\ul,
\us)\right)  \in \sum_{(\um,\,|\up|)\ge(\ul-\ul^+,\, |\us| + |\us^+|)}
\bu(\lien_{i_0-1/2}^-[t])\,\bxur^- (\um,\up).$$
\end{lemma}
\begin{proof}
The second statement follows from the first by using the observation (see \cite[Proposition 3.1.4]{CL}) that  the element
$\bpr\left(\bx^+_{i,r}(\ell_i^+,\bos_i^+)\,\bx^-_{i,r}(\ell_i,\bos_i)\right)$, for $1\le i\le r$,
is a linear combination of elements of the form
$\bx^-_{i,r}(\ell_i-\ell_i^+,\bop_i)$, where $\bop_i$ satisfies the
condition  $|\bop_i| = |\bos_i|+ |\bos_i^+|$.

We now prove the first statement by induction on $|\ul^+| =
\ell^+_1 + \dots +\ell^+_r$. In the case $|\ul^+| =1$, we have $\bxur^+ (\ul^+, \us^+) = x_{i_0,r}^+ \otimes t^s$ for some $s>0$. Observe that
\begin{eqnarray*}&\bpr\left( (x_{i_0,r}^+ \otimes t^s)  \,\bxur^-(\ul, \us)
\right) &= \bpr\left([x_{i_0,r}^+ \otimes t^s, \,\bxur^-(\ul,
\us)]\right)\\ && =\sum_{j=1}^r\bpr\left(g_j \,[x_{i_0,r}^+ \otimes
t^s, \,\bx^-_{j,r} (\ell_j, \bos_j)]\,g_j'\right),\end{eqnarray*}
where $$g_j=\bx^-_{1,r}(\ell_1,\bos_1)\cdots\bx^-_{j-1,r}(\ell_{j-1},\bos_{j-1})\qquad\textup{and}\qquad
 g_j'=\bx^-_{j+1,r}(\ell_{j+1},\bos_{j+1})\cdots\bx^-_{r,r}(\ell_{r},\bos_{r}).$$
As in the proof of the case  $|\ul^+| =1$ in \cite[Lemma 3.5.1]{CL}; we get
$\bpr\left(g_j \,[x_{i_0,r}^+ \otimes t^s, \,\bx^-_{j,r} (\ell_j, \bos_j)]\,g_j'\right)=0,
$ when $i_0<j$, and, in the case $i_0=j$, it is a linear combination of elements of the form
$$g_{i_0} \,\bpr ( (x_{i_0,r}^+\otimes t^s)\,\bx^-_{i_0,r} (\ell_{i_0}, \bos_{i_0}))\,g'_{i_0}\qquad\textup{and}\qquad 
\bxur^-(\ul-\mathbf{e}_{i_0}, \up)\quad\textup{with}\quad |\up|\eqd |\us|+s\mathbf{e}_{i_0}.$$
We now consider the case when $i_0>j$. In this case, we have
$$[x^+_{i_0,r}, \,x^-_{j,r}]=x^-_{j,i_0-1}, \qquad 
[x^-_{p,r}, \,x^-_{j,i_0-1}]=\delta_{i_0,r}\,x^-_{p,\overline{j}}, \,\,\forall\,\,p\leq j, 
\qquad\textup{and}\qquad  [x^-_{q,r}, \,x^-_{p,\overline{j}}]=0,\,\,\forall\,\,q\leq p.
$$
This implies that $\bpr\left(g_j\, [x_{i_0,r}^+ \otimes t^s,\, \bx^-_{j,r} (\ell_j, \bos_j)]\,g_j'\right)$ is equal to
\begin{align*}
&g_j\sum_{m=1}^{\ell_j}\left(\prod_{n=1}^{m-1} x^-_{j,r}\otimes t^{s_j(n)}\right)(x^-_{j,i_0-1}\otimes t^{s+s_j(m)})
\left(\prod_{n=m+1}^{\ell_j}x^-_{j,r}\otimes t^{s_j(n)}\right)
 g_{j'}\\
&=\sum_{m=1}^{\ell_j}\left((x^-_{j,i_0-1}\otimes t^{s+s_j(m)})\,\bxur^-(\um,\up)
+\delta_{i_0,r}\,\sum_{p=1}^j\sum_{m'=1}^{\ell_p} (x^-_{p,\overline{j}}\otimes t^{s+s_j(m)+s_p(m')})\,
\bxur^-(\un,\uq)\right),
\end{align*}
for some $(\um,\up),(\un,\uq)\in \bF^r$ with $\um =\ul-\mathbf{e}_j$ 
and $\un=\ul-\mathbf{e}_j-\mathbf{e}_p$. 
Since $1\leq p\leq j<i_0$, it is clear that $\um,\un\eql \ul-\mathbf{e}_{i_0}$ and 
$x^-_{j,i_0-1}, x^-_{p,\overline{j}}\in \lien^-_{i_0-1/2}$.

For the inductive step, we write 
$\bxur^+ (\ul^+, \us^+)=(x^+_{i_0,r}\otimes t^s)\,\bxur^+ (\um,\up)$ for some $(\um,\up)\in \bF_+^r$ and $s>0$ such that $\um=\ul^+-\mathbf{e}_{i_0}$ and $|\up|=|\us^+|-s$.
Using Proposition~\ref{proppr}~\eqref{proppr3}, we have
\begin{align*}\bpr\left((x^+_{i_0,r}\otimes t^s)\,\bxur^+ (\um,\up)\,
\bxur^-(\ul, \us)\right) &= \bpr\left((x^+_{i_0,r}\otimes
t^s)\,\bpr\left(\bxur^+ (\um,\up) \,\bxur^-(\ul, \us)\right)
\right).
\end{align*}
The result now follows by first using the induction hypothesis for $\bpr\left(\bxur^+ (\um,\up) \,\bxur^-(\ul, \us)\right)
$ and then since $[x^+_{i_0,r}, \lien^-_{i_0'-1/2}]=0,\,\,\forall\,\,i_0'\leq i_0$, the result when $|\ul^+|=1$.
 \end{proof}

\begin{lemma}\label{lem_main}
 Let $(\ul, \us) \in \bF^r_+$ and let $g\in\bu(\lien^{-}_{r-1/2}[t])_{-\eta}, \,\eta\neq0$. Then, $[\bxur^+ (\ul, \us),\, g]$ is a linear combination of elements
of the form $g' \,\bxur^+ (\um,\up)$, where  $g' \in \bu(\lien^-_{r-1/2}[t])$, $(\um,\up) \in \bF^r_+$. Moreover, if $g'$ is a constant then $|\up| \eqd |\us|.$
\end{lemma}

\begin{proof}
It suffices to prove the result when $g$ is a product of terms of
the form $x_{j,k}^-\otimes t^s$, $x_{j,\kbar}^-\otimes t^s$, where $1\le j\le k< r$ and
$s\ge 0$. We prove this by induction on the number $m$ of terms in
the product.

 If $g = x^-_{j,k} \otimes t^s$, $1\le j\le k<r$, then for $i\in I$,
$[x^+_{i,r} \otimes t^{s'}, g]$  is a scalar multiple of $\delta_{i,j}\, x^+_{k+1,r}
\otimes t^{s+s'}$. Hence $[\bxur^+
(\ul, \us), \,g]$ is a linear combination of terms of the form $
\bxur^+ (\um,\up)$, where $(\um,\up) \in \bF^r_+$ and $$|\up|
= |\us| - \bos_j(m)\boe_j + (\bos_j(m)+s)\boe_{k+1},\quad 1\le
m\le \ell_j.$$ Since $\bos_j(m)>0,\,\, \forall\,\, 1\le m\le \ell_j$, we have
$|\up| \eqd |\us|$.

If $g = x^-_{j,\kbar} \otimes t^s$, $1\le j\le k<r$, then for $i\in I$,
\begin{gather*}
[x^+_{i,r} \otimes t^{s'}, g]= \begin{cases} x^-_{k,r-1}\otimes t^{s+s'}, & i=j,\\
x^-_{j,r-1}\otimes t^{s+s'}, & i=k,\\ 0, & \textup{otherwise.}\end{cases}
\end{gather*}
Since $$[x_{i,r}^+, x_{k,r-1}^-]=-2\delta_{i,k}\, x_{r,r}^+, \,\,\forall\,\,i\leq j\qquad\textup{and}\qquad [x_{i,r}^+,\, x_{j,r-1}^-]=-2\delta_{i,j}\, x_{r,r}^+, \,\,\forall\,\,i\leq k,$$
$[\bxur^+(\ul, \us), \,g]$ is a linear combination of terms of the form $ \bxur^+ (\um,\up), \,(x^-_{k, r-1}\otimes t^{s+s_j(m)})\,\bxur^+ (\un,\uq)$, and
$(x^-_{j, r-1}\otimes t^{s+s_k(n)})\,\bxur^+ (\uk,\uu)$, where $(\um,\up), (\un,\uq), (\uk,\uu) \in \bF^r_+$ with
$$|\up|= |\us| - \bos_j(m')\boe_j - \bos_k(n')\boe_{k}+(\bos_j(m')+\bos_k(n')+s)\boe_{r},$$ 
$1\le
m,m'\le \ell_j, 1\le n,n'\le \ell_k.$ 
Since $\bos_j(m')>0,\,\, \forall\,\, 1\le m'\le \ell_j$, we have
$|\up| \eqd |\us|$, hence induction begins.

For the inductive step, write   $g = yg_0$,
where $y=x_{j,k}^-\otimes t^s$ or $x_{j,\kbar}^-\otimes t^s$ and $g_0$ is a product of $(m-1)$--elements of the form
$x_{j,k}^-\otimes t^{s}, x_{j,\kbar}^-\otimes t^s$, where $1\le j\le k< r$ and $s\geq 0$ . We have $$[\bxur^+ (\ul, \us),\, y g_0] =
[\bxur^+ (\ul, \us),\, y]\,g_0 + y\,[\bxur^+ (\ul, \us),\, g_0].$$ The second term has the
required form by the induction hypothesis. 
From the $m=1$ case, the first term is a linear combination of elements
of the form $g' \,\bxur^+ (\um,\up)\,g_0$, where  $g' \in \bu(\lien^-_{r-1/2}[t])$ and $(\um,\up) \in \bF^r_+$ with  $|\up| \eqd |\us|$ when $g'$ is a constant.
Since 
$$g'\, \bxur^+ (\um,\up)\,g_0=g'\,[\bxur^+ (\um,\up),\, g_0]+g' g_0 \,\bxur^+ (\um,\up),$$
the result follows by using the induction hypothesis again.
\end{proof}
The following result is immediate from \cite[Proposition 3.4.4]{CL}.
\begin{proposition}\label{prop_sl2} Let $1\le i\le r$ and $n\ge 0$.
For all $(\ell,\bos)\in\bF$, the element  $\bx^-_{i,r}
(\ell,\bos)$ is in the span of the union of
$$\left\{\bx^-_{i,r}(\ell,\bop):(\ell,\bop)\in \bF(n), \
|\bop|=|\bos|\right\},$$ and $$
 \left\{\bpr(\bx^+_{i,r}(m-\ell,\bop)\,(x^-_{i,r} \otimes 1)^{m}): m>n,\
(m-\ell,\bop)\in\bF_+,\   |\bop|=|\bos| \right\}.$$ In
particular, the element $\bx_{i,r}^-(\ell,\bos)$ is in the span
of
 $$
 \left\{\bpr(\bx^+_{i,r}(m-\ell,\bop)\,(x^-_{i,r} \otimes 1)^{m}): m\ge 0,\
(m-\ell,\bop)\in\bF_+, \   |\bop|=|\bos| \right\}.$$
\end{proposition}

The proof of the following proposition is analogous to that of \cite[Proposition 3.6.3]{CL} and it uses Propositions~\ref{prop_fd},~\ref{proppr}~\eqref{prl}, and \ref{prop_sl2},  and Lemmas~\ref{lem_dif}--\ref{lem_main}.

\begin{proposition}\label{prop_bound1} Let
$\lambda=\sum_{i=1}^rm_i\omega_i\in P^+$.  Fix $1\le k\le r$ and
$m\in\bn$ with $m>m_k$. Let $(\ell_i, \bop_i) \in \bF$, $1\le i
\ne k \le r$, and $(m-\ell_k, \bop_k) \in \bF_+$.  Then, we have
\begin{align*}
&\prod_{i=1}^{k-1} \bx^-_{i,r}(\ell_i, \bop_i)\,
\bpr\left(\bx_{k,r}^+(m-\ell_k,\bop_k)\, (x_{k,r}^- \otimes
1)^{m}\right)  \prod_{i=k+1}^{r} \bx^-_{i,r}(\ell_i, \bop_i)\,
w_\lambda \\&\in 
\sum_{\{(\un,\,\uq)\in\bF^r:(\un,\,|\uq|)>(\ul,\,|\up|)\}} \bu(\lien^-_{r-1/2}[t])\,\bxur^-(\un,\uq)\,w_\lambda.
\end{align*}
\end{proposition}

\subsection{} {\it Proof of Proposition~\ref{mainprop}~\eqref{mainpropp1}}.  Let
$\lambda=\sum_{i=1}^r m_i\omega_i\in P^+$. Using
Proposition~\ref{prop_sl2}  simultaneously for $1\le i\le r $ and
for $n = m_1, \dots, m_r$, we see that the element $\bxur^-(\ul,\us)$, for $(\ul,\us)\in\bF^r$, is in the span of the elements from
\begin{equation}\label{these}
\bxur^-(\ul,\up), \qquad  |\up| = |\us|,
\ \
(\ul,\up) \in \bF^r(\lambda)
\end{equation}
 together with the elements from
\begin{equation}\label{other}
\prod_{i=1}^{k-1}\bx^-_{i,r}(\ell_i,\bop_i)\,
\bpr\left(\bx^+_{k,r}(m-\ell_k,\bop_k)\,(x_{k,r}^-\otimes 1)^m\right)
\prod_{i=k+1}^{r}\bx^-_{i,r}(\ell_i,\bop_i), \qquad m> m_k, \ \,
\ \ |\up| = |\us|,
\end{equation}
where  $(\ell_i, \bop_i) \in \bF$,
$1\le i \ne k \le r$, and $(m-\ell_k, \bop_k) \in \bF_+$.
Now the proof follows by using Proposition~\ref{prop_bound1}.

\bibliographystyle{bibsty-final-no-issn-isbn}
\addcontentsline{toc}{section}{References}

\end{document}